\documentclass[10pt] {amsart}

\usepackage{geometry}
\usepackage{multirow}
\usepackage{comment}
\usepackage{color}
\usepackage{amssymb}
\usepackage{amsmath}
\usepackage{amsthm}
\usepackage{amsfonts}
\usepackage{tikz}
\usetikzlibrary{matrix}
\usepackage[latin1]{inputenc}

\usetikzlibrary{arrows}

\newcommand\restr[2]{{
  \left.\kern-\nulldelimiterspace 
  #1 
  \vphantom{\big|} 
  \right|_{#2} 
  }}

\def\nicefrac#1#2{\leavevmode%
    \raise.5ex\hbox{\tiny #1}%
    \kern-.1em/\kern-.15em%
    \lower.25ex\hbox{\tiny #2}}
  
\newcommand\slimminus{\hspace{-1pt}-\hspace{-1pt}}
\newcommand\vslimminus{\hspace{-2pt}-\hspace{-2pt}}
\newcommand\vslimplus{\hspace{-2pt}+\hspace{-2pt}}
 
\newtheorem{theorem}{Theorem}[section]
\newtheorem{lemma}[theorem]{Lemma}
\newtheorem{proposition}[theorem]{Proposition}
\newtheorem{corollary}[theorem]{Corollary}

\theoremstyle{definition}
\newtheorem{definition}[theorem]{Definition}
\newtheorem{remark}[theorem]{Remark}
\newtheorem{example}[theorem]{Example}

\definecolor{gray1}{rgb}{0.22,0.22,0.22}
\definecolor{gray2}{rgb}{0.28,0.28,0.28}
\definecolor{gray3}{rgb}{0.36,0.36,0.36}
\definecolor{gray4}{rgb}{0.44,0.44,0.44}
\definecolor{gray5}{rgb}{0.52,0.52,0.52}
\definecolor{gray6}{rgb}{0.6,0.6,0.6}
\definecolor{gray7}{rgb}{0.68,0.68,0.68}
\definecolor{gray8}{rgb}{0.76,0.76,0.76}

\begin{document}

\title{Torus Invariant Curves}

\author{Geoffrey Scott}
\thanks{The author was partially supported by NSF RTG grant DMS-1045119}
\thanks{The author was partially supported by NSF RTG grant DMS-0943832}
\email{gsscott@umich.edu}

\begin{abstract}

Using the language of T-varieties, we study torus invariant curves on a complete normal variety $X$ with an effective codimension-one torus action. In the same way that the $T$-invariant Weil divisors on $X$ are sums of ``vertical'' divisors and ``horizontal'' divisors, so too is each $T$-invariant curve a sum of ``vertical'' curves and ``horizontal'' curves. We give combinatorial formulas that calculate the intersection between $T$-invariant divisors and $T$-invariant curves, and generalize the celebrated toric cone theorem to the case of complete complexity-one $T$-varieties.

\end{abstract}

\maketitle

\section{Introduction}
A $T$-variety is a normal complex algebraic variety with an effective action of an algebraic torus. This definition matches the definition of a toric variety, except that the dimension of the torus may be less than the dimension of the variety on which it acts. In particular, any normal algebraic variety is a $T$-variety when endowed with the trivial action of $(\mathbb{C}^*)^0$. We therefore can't expect to prove much about general $T$-varieties; we usually restrict our attention to {\it complexity-one} $T$-varieties, where the dimension of the torus is exactly one less than the dimension of the variety. In this paper, we study the $T$-invariant curves of a complete complexity-one $T$-variety, find formulas for their intersection with $T$-invariant divisors (using the theory of $T$-invariant divisors developed by Petersen and S\"u\ss \ in \cite{tid}), and prove that the numerical equivalence classes of these curves generate the Mori cone of the $T$-variety.

We review the basics of $T$-varieties in Section \ref{sec_primer}. Informally speaking, a complexity-one $T$-variety is encoded by a family (parametrized by a projective curve $Y$) of polyhedral subdivisions of a vector space, all with the same tailfan. In Section \ref{sec_curves}, we describe two kinds of $T$-invariant curves in a $T$-variety, {\it vertical} curves and {\it horizontal} curves. The vertical curves correspond to walls (codimension-one strata) of one of these polyhedral subdivisions, while the horizontal curves correspond to certain maximal-dimensional cones of the tailfan. We give formulas that calculate the intersection of these curves with a $T$-invariant divisor using the language of Cartier support functions from \cite{tid}.

In Section \ref{sec_tct}, we generalize the toric cone theorem, which states that the Mori cone of a toric variety is generated as a cone by the classes of $T$-invariant curves corresponding to the walls of its fan. In our generalization, we show that the Mori cone of a complete complexity-one $T$-variety is generated as a cone by the classes of a finite collection of vertical curves and horizontal curves. We end the paper with examples in Section \ref{sec_examples}.

\section{Primer on $T$-varieties}\label{sec_primer}

In this section, we review the basic notation and construction of $T$-varieties. The presentation favors brevity over pedogogy; we encourage any reader unfamiliar with $T$-varieties to read the excellent survey article \cite{gotv} for a friendlier exposition to this beautiful topic.

\subsection{Notation}
Let $T \cong (\mathbb{C}^*)^k$ be an algebraic torus, and $M, N$ be the character lattice of $T$ and the lattice of 1-parameter subgroups of $T$ respectively. These lattices embed in the vector spaces
\[
N_{\mathbb{Q}} := \mathbb{Q} \otimes N \hspace{2cm} M_{\mathbb{Q}} := \mathbb{Q} \otimes M
\]
and are dual to one another\footnote{In this paper, when a picture of $N_{\mathbb{Q}}$ is juxtaposed with a picture of $M_{\mathbb{Q}}$, the reader may assume that the bases for these vector spaces have been chosen so that the pairing between them is the standard dot product.}. In classic toric geometry, one studies the correspondence between cones (and fans) in $N_{\mathbb{Q}}$ and the toric varieties encoded by these combinatorial data. Analogously, we study $T$-varieties through the correspondence between combinatorial gadgets called $p$-divisors (and divisoral fans) and the $T$-varieties they encode. Informally speaking, a $p$-divisor is a Cartier divisor on a semiprojective variety $Y$ with polyhedral coefficients; a divisorial fan is a collection of $p$-divisors whose polyhedral coefficients ``fit together'' in a suitable way. To make formal these definitions, we begin by discussing monoids of polyhedra.

Let $\sigma$ be a pointed cone in $N_{\mathbb{Q}}$, and $\sigma^{\vee} \subseteq M_{\mathbb{Q}}$ its dual. The set $\textrm{Pol}_{\mathbb{Q}}^+(N, \sigma)$ of all polyhedra in $N_{\mathbb{Q}}$ having $\sigma$ as its tailcone (with the convention that $\emptyset \in \textrm{Pol}_{\mathbb{Q}}^+(N, \sigma)$) is a monoid under Minkowski addition with identity element $\sigma$. Any nonempty $\Delta \in \textrm{Pol}_{\mathbb{Q}}^+(N, \sigma)$ defines a map
\begin{align}\label{eqn_evaluate_polyhedron}
h_{\Delta}: \sigma^\vee &\rightarrow \mathbb{Q}\\
u &\mapsto \textnormal{min}_{v \in \Delta} \langle v, u \rangle \notag
\end{align}
called the {\it support function} of $\Delta$. A nonempty $\Delta \in \textrm{Pol}_{\mathbb{Q}}^{+}(N, \sigma)$ also defines a {\it normal quasifan} $\mathcal{N}(\Delta)$ in $M_{\mathbb{Q}}$ consisting of a cone $\lambda_F$ for each face $F$ of $\Delta$ defined by
\[
\lambda_F = \{u \in \sigma^{\vee} \mid \langle u, v \rangle = h_{\Delta}(u) \ \forall v \in F\}.
\]
The figure below shows an example of a polyhedron and its normal quasifan.

\begin{figure}[!ht]
\centering
\begin{tikzpicture}[scale = 0.6]

\path[shift = {(0, -1)}, fill = gray5] (3.1, 3.1) -- (-0.4, 3.1) -- (-1.5, 2) -- (-1.5, 1) -- (0.5, 0) -- (3.1, 0) -- cycle;

\draw[very thick, shift = {(0, -1)}, ->] (0.5, 0) -- (3.2, 0);
\draw[very thick, shift = {(0, -1)}, ->] (0.5, 0) -- (-1.5, 1) -- (-1.5, 2) -- (-0.3, 3.2);

\path[shift = {(9, 0)}, fill = gray2] (0, 0) -- (0, 3.1) -- (1.55, 3.1) -- cycle;
\path[shift = {(9, 0)}, fill = gray5] (0, 0) -- (1.55, 3.1) -- (3.1, 3.1) -- (3.1, 0) -- cycle;
\path[shift = {(9, 0)}, fill = gray8] (0, 0) -- (3.1, 0) -- (3.1, -3.1) -- cycle;

\foreach \x in {-3, ..., 3} { \foreach \y in {-3, ..., 3} { 
	\draw[fill = black, shift = {(9, 0)}] (\x, \y) circle(.7mm); 
	}}

\draw[very thick, ->, shift = {(9, 0)}] (0, 0) -- (3.2, 0);
\draw[very thick, ->, shift = {(9, 0)}] (0, 0) -- (3.2, -3.2);
\draw[very thick, ->, shift = {(9, 0)}] (0, 0) -- (1.6, 3.2);
\draw[very thick, ->, shift = {(9, 0)}] (0, 0) -- (0, 3.2);

\draw (0, -4) node[] {$\Delta$};
\draw (9, -4) node[] {$\mathcal{N}(\Delta)$};

\end{tikzpicture}
\caption{A polyhedron and its normal quasifan}
\end{figure}
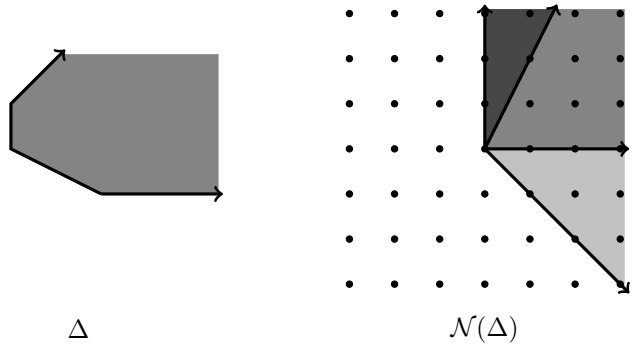

\begin{proposition}\label{prop_suppt_fxns}(\cite{poly_divisors}, Lemma 1.4 and Proposition 1.5)
The support function $h_{\Delta}$ is a well-defined map whose regions of linearity are the maximal cones of $\mathcal{N}(\Delta)$. Moreover, any function in $\textrm{Hom}(\sigma^\vee, \mathbb{Q})$ whose regions of linearity define a quasifan can be realized as $h_{\Delta}$ for some $\Delta$.
\end{proposition}

Let $\textrm{Pol}_{\mathbb{Q}}(N, \sigma)$ be the Grothendieck group of $\textrm{Pol}_{\mathbb{Q}}^+(N, \sigma)$. Let $Y$ be a semiprojective variety, with $\textrm{CaDiv}(Y)$ its group of Cartier divisors. An element 
\[
\mathcal{D} \in \textrm{Pol}_{\mathbb{Q}}(N, \sigma) \otimes_{\mathbb{Z}} \textrm{CaDiv}(Y)
\]
is a {\it polyhedral divisor} with tailcone $\sigma$ if it has a representative of the form $\mathcal{D} = \sum\mathcal{D}_P \otimes P$ for some $\mathcal{D}_P \in \textrm{Pol}_{\mathbb{Q}}^+(N, \sigma)$ and $P$ prime\footnote{Because $\sigma$ (not $\emptyset$) is the identity element of $\textrm{Pol}_{\mathbb{Q}}(N, \sigma)$, the summation notation in this sentence implies that only finitely many of the polyhedral coefficients $\mathcal{D}_P$ differ from $\sigma$}. We will describe a procedure for constructing an affine $T$-variety from a certain kind of polyhedral divisor (called a \emph{p-divisor}); this construction will involve taking the spectrum of the global sections of a sheaf of rings defined over a subset of $Y$. This subset, called the {\it locus} of $\mathcal{D}$, is 
\[
\textnormal{Loc}(\mathcal{D}) := Y \backslash \cup_{\mathcal{D}_P = \emptyset} P.
\]
The {\it evaluation} of $\mathcal{D}$ at $u \in M \cap \sigma^\vee$ is the $\mathbb{Q}$-Cartier divisor\footnote{Some authors define a ``$\mathbb{Q}$-Cartier'' divisor to be a Weil divisor with a Cartier multiple. Our $\mathbb{Q}$-Cartier divisors are elements of $\mathbb{Q} \otimes Div(Y)$ having a Cartier multiple (so may have rational coefficients). The pedantic reader is invited to replace all instances of ``$\mathbb{Q}$-Cartier divisor'' in this paper with ``$\mathbb{Q}$-Cartier $\mathbb{Q}$-divisor''.}
\[
\mathcal{D}(u) := \sum_{\mathcal{D}_P \neq \emptyset}h_{\mathcal{D}_P}(u)\restr{P}{\textnormal{Loc}(\mathcal{D})}.
\]
We say that $\mathcal{D}$ is a {\it $p$-divisor} if  $\mathcal{D}(u)$ is semiample for all $u \in \sigma^{\vee}$ and big for all $u$ in the interior of $\sigma^{\vee}$. The direct sum of the sheaves defined by the evaluations $\mathcal{D}(u)$ is an $M$-graded sheaf of rings
\[
\mathcal{O}(\mathcal{D}) := \bigoplus_{u \in \sigma^\vee \cap M} \mathcal{O}_{\textnormal{Loc}(\mathcal{D})}(\mathcal{D}(u))\chi^u
\]
over $\textnormal{Loc}(\mathcal{D})$. There are two different $T$-varieties encoded by the $p$-divisor $\mathcal{D}$
\[
\widetilde{TV}(\mathcal{D}) := \textnormal{Spec}_{\textnormal{Loc}(\mathcal{D})}\mathcal{O}(\mathcal{D}) \hspace{1cm} \textrm{and} \hspace{1cm} TV(\mathcal{D}) := \textnormal{Spec} \  \Gamma(\textnormal{Loc}(\mathcal{D}), \mathcal{O}(\mathcal{D}))
\]
where the torus action is given by the $M$-grading on $\mathcal{O}(\mathcal{D})$. All affine $T$-varieties can be constructed this way.
\begin{theorem}(\cite{poly_divisors}, Corollary 8.14) Every normal affine variety with an effective torus action can be realized as $TV(\mathcal{D})$ for some $p$-divisor $\mathcal{D}$
\end{theorem}
Similar to the way that a fan of a non-affine toric variety can be obtained by ``gluing together'' the cones constituting an affine cover, so too can a non-affine $T$-variety be encoded by ``gluing together'' the $p$-divisors constituting an affine cover. To make formal these concepts, we first define the {\it intersection} of two $p$-divisors $\mathcal{D}, \mathcal{D}'$ on $Y$ as the $p$-divisor
\[
\mathcal{D} \cap \mathcal{D}' := \sum (\mathcal{D}_P \cap \mathcal{D}_P') \otimes P\textrm{.}
\]
We say that $\mathcal{D}'$ is a {\it face} of $\mathcal{D}$ if $\mathcal{D}'_P \subseteq \mathcal{D}_P$ for each $P$ and the induced map $TV(\mathcal{D}') \rightarrow TV(\mathcal{D})$ is an open embedding. A finite set $\mathcal{S}$ of $p$-divisors on $Y$ is a {\it divisoral fan} if for any $\mathcal{D}, \mathcal{D}' \in \mathcal{S}$, $\mathcal{D} \cap \mathcal{D}'$ is an element of $\mathcal{S}$ and is a face of both $\mathcal{D}$ and $\mathcal{D}'$. We define $TV(\mathcal{S})$ and $\widetilde{TV}(\mathcal{S})$ to be the $T$-varieties obtained by gluing together the $T$-varieties $\{TV(\mathcal{D})\}_{\mathcal{D} \in \mathcal{S}}$ and $\{\widetilde{TV}(\mathcal{D})\}_{\mathcal{D} \in \mathcal{S}}$ according to these face relations. This process is detailed in \cite{fan_gluing}.
\subsection{Geometry of $TV(\mathcal{S})$ and $\widetilde{TV}(\mathcal{S})$}
Because $\widetilde{TV}(\mathcal{D})$ is defined as the relative spectrum of a sheaf of rings on $\textnormal{Loc}(\mathcal{D})$, there is a natural projection map $\pi: \widetilde{TV}(\mathcal{D}) \rightarrow \textnormal{Loc}(\mathcal{D}) \subseteq Y$. Because $TV(\mathcal{D})$ is defined as the spectrum of the global sections of the structure sheaf on $\widetilde{TV}(\mathcal{D})$, we also have a natural map $p: \widetilde{TV}(\mathcal{D}) \rightarrow \Gamma(\widetilde{TV}(\mathcal{D}), \mathcal{O}_{\widetilde{TV}(\mathcal{D})}) \cong TV(\mathcal{D})$. Given a divisoral fan $\mathcal{S}$, the maps $\pi, p$ corresponding to the different $\mathcal{D} \in \mathcal{S}$ glue into maps
\begin{figure}[!ht]
\centering
\begin{tikzpicture}
\draw (0, 1.5) node (box1){$\widetilde{TV}(\mathcal{S})$};

\draw (3, 1.5) node (box2){${TV}(\mathcal{S})$};

\draw (0, 0) node (box3){$Y$};

\draw[->] (0, 1.25) -- node[right] {$\pi$} (0, 0.25);
\draw[->] (0.6, 1.5) -- node[above] {$p$} (2.4,1.5);
\end{tikzpicture} 
\end{figure}

In this subsection, we describe the fibers of $p$ and $\pi$. In particular, we will notice that for $y \in Y$, the reduced fiber $\pi^{-1}(y)$ is a union of irreducible toric varieties, and that the contraction map $p$ identifies certain disjoint torus orbits of $\widetilde{TV}(\mathcal{S})$. Many of these results simplify when $TV(\mathcal{S})$ is a complexity-one $T$-variety; because this is the only case we will need for later sections, we will henceforth assume that $Y$ is a projective curve. The reader interested in higher-complexity $T$-varieties should read \cite{gotv} for the more general results.

In \cite{p_thesis}, the author describes the reduced fibers of $\pi$ using the language of {\it dappled toric bouquets}. We begin by reviewing this language. 
\begin{definition}
The {\it fan ring} of a quasifan $\Lambda$ in $M_{\mathbb{Q}}$ is
\[
\mathbb{C}[\Lambda] := \bigoplus_{u \in |\Lambda| \cap M} \mathbb{C}\chi^u
\]
with multiplication defined by
\[
\chi^u\chi^v = \left\{ 
\begin{array}{c c}
\chi^{u+v} & \ \textnormal{if} \ u, v \in \lambda \ \textnormal{for some cone} \ \lambda \in \Lambda \\
0          & \ \textrm{otherwise}
\end{array}\right.
\]
\end{definition}

For a nonempty $\Delta \in \textnormal{Pol}_{\mathbb{Q}}^+(N, \sigma)$ and a cone $\lambda_F$ of its inner normal quasifan $\mathcal{N}(\Delta)$, let 
\[
M_{\lambda_F} := \{u \in \lambda_F \cap M \mid h_{\Delta}(u) \in \mathbb{Z}\}.
\]
\begin{remark}
In other papers, $M_{\lambda_F}$ is defined differently: when $\Delta\otimes [P]$ appears as a summand in a $p$-divisor, the elements $u \in M_{\lambda_{F}}$ are required to satisfy the condition that $h_{\Delta}(u) [P]$ is locally principal at $P$. In the complexity-one case, this condition coincides with our condition that $h_{\Delta}(u) \in \mathbb{Z}$.
\end{remark}

Finally, let $S_{\Delta} \subseteq |\Lambda(\Delta)| \cap M$ consist of those $u$ such that $S_{\Delta} \cap \lambda_F = M_{\lambda_F}$ for every cone $\lambda_F \in \mathcal{N}(\Delta)$. $S_{\Delta}$ can be thought of as a conewise-varying sublattice of $M$. The figure below shows an example of $S_{\Delta}$ for a given $\Delta$; the elements of $S_{\Delta} \subseteq M$ are in bold.

\begin{figure}[!ht]
\centering
\begin{tikzpicture}[scale = 0.3]

\path[shift = {(-7, 0)}, fill = gray5] (4.1 * 3, 4.1 * 3) -- (2.1 * 3, 4.1 * 3) -- (0.5 * 3, 2.5 * 3) -- (0.5 * 3, 1.334 * 3) -- (2 * 3, .58334 * 3) -- (4.1 * 3, 0.58334 * 3) -- cycle;

\draw[very thick, shift = {(-7, 0)}] (0.5 * 3, 2.5 * 3) -- (0.5 * 3, 1.334 * 3) -- (2 * 3, .58334 * 3) -- (4.2 * 3, 0.58334 * 3);
\draw[very thick, shift = {(-7, 0)}] (0.5 * 3, 2.5 * 3) -- (2.2 * 3, 4.2 * 3);

\foreach \x in {-1, ..., 4} { \foreach \y in {-1, ..., 4} { 
	\draw[fill = black, shift = {(-7, 0)}] (\x * 3, \y * 3) circle(1mm); 
	}}
	
\draw[ shift = {(-7, 0)}] (0, -1.2 * 3) -- (0, 4.2 * 3); \draw[ shift = {(-7, 0)}] (-1.2 * 3, 0) -- (4.2 * 3, 0);

\draw[ shift = {(-7, 0)}] (0.5 * 3, 0.2) -- (0.5 * 3, -0.2) node[below] {$\frac{1}{2}$};
\draw[ shift = {(-7, 0)}] (0.2, 0.58334*3) -- (-0.2, 0.58334 * 3) node[left] {$\frac{7}{12}$};
\draw[ shift = {(-7, 0)}] (0.2, 1.334 * 3) -- (-0.2, 1.334 * 3) node[left] {$\frac{4}{3}$};
\draw[ shift = {(-7, 0)}] (0.2, 2.5 * 3) -- (-0.2, 2.5 * 3) node[left] {$\frac{5}{2}$};

\foreach \x in {-1, ..., 12} { \foreach \y in {-4, ..., 14} { 
	\draw[fill = black, shift = {(9, 0)}] (\x, \y) circle(.7mm); 
	}}

\foreach \x in {0, 1, ..., 6} { \draw[fill = black, shift = {(9, 0)}] (\x, 12) circle(1.5mm); }
\foreach \x in {6, 8, ..., 12} { \draw[fill = black, shift = {(9, 0)}] (\x, 12) circle(1.5mm); }
\foreach \x in {6, 8, ..., 12} { \draw[fill = black, shift = {(9, 0)}] (\x, 9) circle(1.5mm); }
\foreach \x in {4, 6, ..., 12} { \draw[fill = black, shift = {(9, 0)}] (\x, 6) circle(1.5mm); }
\foreach \x in {2, 4, ..., 12} { \draw[fill = black, shift = {(9, 0)}] (\x, 3) circle(1.5mm); }	
\foreach \x in {0, 2, ..., 12} { \draw[fill = black, shift = {(9, 0)}] (\x, 0) circle(1.5mm); }
\foreach \x in {1, 3, ..., 12} { \draw[fill = black, shift = {(9, 0)}] (\x, -1) circle(1.5mm); }
\foreach \x in {2, 4, ..., 12} { \draw[fill = black, shift = {(9, 0)}] (\x, -2) circle(1.5mm); }
\foreach \x in {3, 5, ..., 12} { \draw[fill = black, shift = {(9, 0)}] (\x, -3) circle(1.5mm); }
\foreach \x in {4, 6, ..., 12} { \draw[fill = black, shift = {(9, 0)}] (\x, -4) circle(1.5mm); }

\draw[thick, ->, shift = {(9, 0)}] (0, 0) -- (12.2, 0);
\draw[thick, ->, shift = {(9, 0)}] (0, 0) -- (4.2, -4.2);
\draw[thick, ->, shift = {(9, 0)}] (0, 0) -- (7.1, 14.2);
\draw[thick, ->, shift = {(9, 0)}] (0, 0) -- (0, 14.2);

\draw (-4, -6) node[] {$\Delta \subseteq N_{\mathbb{Q}}$};
\draw (15, -6) node[] {$S_{\Delta} \subseteq M$};
\end{tikzpicture}
\caption{$S_{\Delta}$ is a conewise-varying sublattice of $M$}
\end{figure}
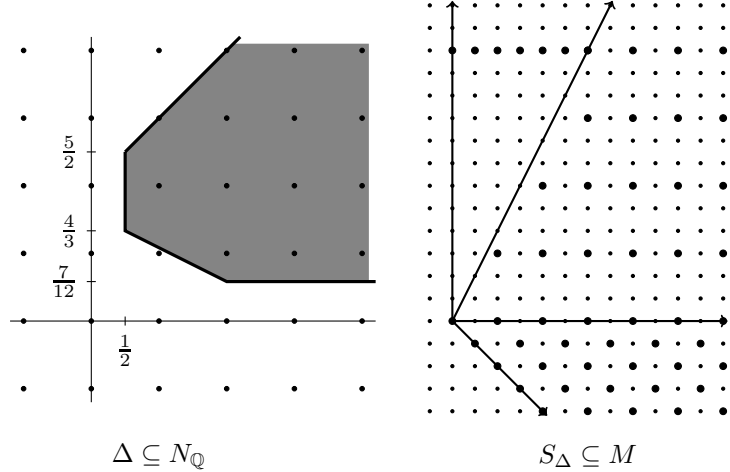

\begin{definition}
The {\it dappled fan ring} of $\Delta$ is the following subring of $\mathbb{C}[\mathcal{N}(\Delta)]$
\[
\mathbb{C}[\mathcal{N}(\Delta), S_{\Delta}] := \bigoplus_{u \in S_{\Delta}} \mathbb{C}\chi^u
\]
\end{definition}
\begin{definition}
The {\it dappled toric bouquet} encoded by $\Delta$ is the variety $TB(\Delta) := \textnormal{Spec}(\mathbb{C}[\mathcal{N}(\Delta), S_{\Delta}])$. Given a polyhedral complex $\Sigma = \{\Delta\}$ in $N_{\mathbb{Q}}$, the dappled toric bouquet encoded by $\Sigma$ is the variety $TB(\Sigma)$ obtained by gluing the $\{TB(\Delta)\}_{\Delta \in \Sigma}$ according to the face relations among the polyhedra.
\end{definition}
Observe that $TB(\Delta)$ and $TB(\Sigma)$ have a natural torus action induced by the $M$-grading of the dappled fan rings. For a $T$-variety $TV(\mathcal{S})$ over $Y$ and a point $y \in Y$, the polyhedra $\{\mathcal{D}_y\}_{\mathcal{D} \in \mathcal{S}}$ fit together into a polyhedral complex $\mathcal{S}_y$ of $N_{\mathbb{Q}}$.
\begin{proposition}\label{prop_tbp} [\cite{p_thesis}, Prop 1.29] Let $\mathcal{S}$ be a p-divisor on the smooth projective curve $Y$. The reduced fiber $\pi^{-1}(y)$ of $\pi: \widetilde{TV}(\mathcal{S}) \rightarrow Y$ is equivariantly isomorphic to $TB(\mathcal{S}_y)$. 
\end{proposition}

Motivated by Proposition \ref{prop_tbp} to study the geometry of non-affine toric bouquets, we construct a fan for each vertex of a polyhedral subdivision $\Sigma$ of $N_{\mathbb{Q}}$; the toric varieties they encode will be precisely the irreducible components of $TB(\Sigma)$. For a vertex $v \in \Sigma$, define the lattice 
\[
M_{v} = \{u \in M \mid \langle u, v \rangle \in \mathbb{Z}\}
\]
Because $M_{v}$ is a sublattice of $M$, $N$ is a sublattice of $N_{v} := M_v^{\vee} \subseteq N_{\mathbb{Q}}$. Let $i_{v}: N_{\mathbb{Q}} \rightarrow (N_v)_{\mathbb{Q}}$ be the map induced by this inclusion. As $\Delta$ ranges over all polyhedra in $\Sigma$ containing $v$, the cones $i_v(\mathbb{Q}_{\geq 0} \cdot (\Delta - v))$ form a fan $F_v$ in $(N_v)_{\mathbb{Q}}$. For any cone $\sigma = i_v(\mathbb{Q}_{\geq 0} \cdot (\Delta - v))$ of $F_v$, the semigroup $\sigma^{\vee} \cap N_v^{\vee}$ is isomorphic to the semigroup $\lambda_{\Delta} \cap S_{\Delta}$. Because this isomorphism commutes with the gluing data induced by the face relations, we have the following description of the irreducible components of $TB(\Sigma)$.
\begin{proposition}\label{prop_irrcpt}
The irreducible components of $TB(\Sigma)$ are equivariantly isomorphic to the toric varieties $\{TV(F_{v})\}$ where the set ranges over the vertices $v$ of $\Sigma$.
\end{proposition}

For example, the polyhedral complex in Figure \ref{fig_polycomp} encodes a toric bouquet with three irreducible toric components. We have drawn the lattices $N_v$ not as a square grid, but in a way that the sublattice $N \subseteq N_v$ (in bold) is a square grid so that the angles between the polyhedra are preserved. In the example, one fan encodes $\mathbb{P}^2$ and the other fans encode weighted projective spaces. 

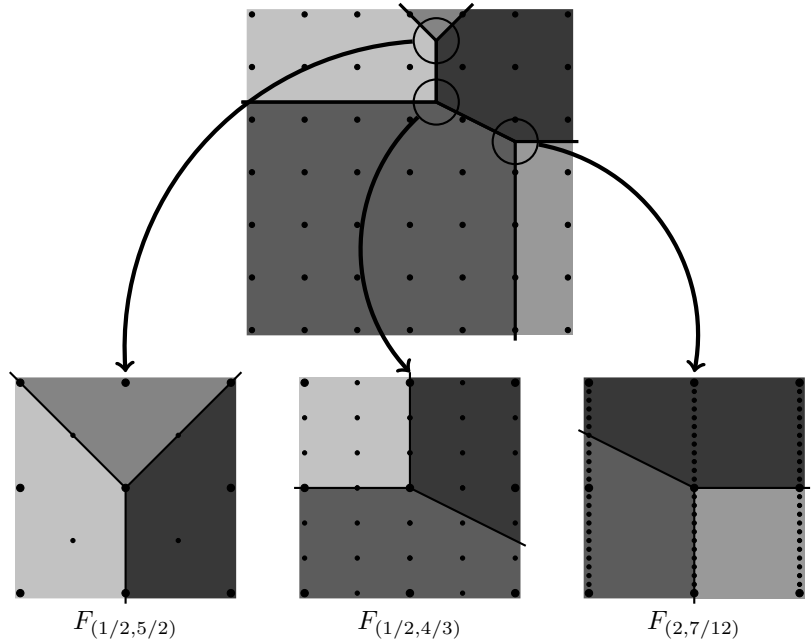
\begin{figure}[!ht]
\centering
\begin{tikzpicture}[scale = 0.7]
\path[shift = {(5, 5)}, fill = gray1] (3.1 , 3.1) -- (1.1 , 3.1) -- (0.5, 2.5) -- (0.5, 1.334 ) -- (2, 0.58334) -- (3.1, 0.58334) -- cycle;
\path[shift = {(5, 5)}, fill = gray5] (-0.1 , 3.1) -- (0.5, 2.5) -- (1.1, 3.1) -- cycle;
\path[shift = {(5, 5)}, fill = gray8] (-0.1 , 3.1) -- (0.5, 2.5) -- (0.5, 1.334 ) -- (-3.1, 1.334) -- (-3.1, 3.1) -- cycle;
\path[shift = {(5, 5)}, fill = gray3] (-3.1, 1.334)  -- (0.5, 1.334 ) -- (2, 0.58334) -- (2, -3.1) -- (-3.1, -3.1) -- cycle;
\path[shift = {(5, 5)}, fill = gray6] (2, -3.1) -- (2, 0.58334) -- (3.1, 0.58334) -- (3.1, -3.1) -- cycle;

\draw[shift = {(5, 5)}, very thick] (1.2, 3.2) -- (0.5, 2.5) -- (0.5, 1.334 ) -- (2, 0.58334) -- (3.2, 0.58334);
\draw[shift = {(5, 5)}, very thick] (-0.2, 3.2) -- (0.5, 2.5) -- (0.5, 1.334 ) -- (-3.2, 1.334);
\draw[shift = {(5, 5)}, very thick] (-3.2, 1.334) -- (0.5, 1.334) -- (2, 0.58334) -- (2, -3.2);

\foreach \x in {-3, ..., 3} { \foreach \y in {-3, ..., 3} { 
	\draw[fill = black, shift = {(5, 5)}] (\x, \y) circle(.5mm); 
	}}

\path[shift = {(-0.4, -1)}, fill = gray5] (2.1, 2.1) -- (0, 0) -- (-2.1, 2.1) -- cycle;
\path[shift = {(-0.4, -1)}, fill = gray8] (-2.1, 2.1) -- (0, 0) -- (0, -2.1) -- (-2.1, -2.1) -- cycle;
\path[shift = {(-0.4, -1)}, fill = gray1] (2.1, 2.1) -- (0, 0) -- (0, -2.1) -- (2.1, -2.1) -- cycle;

\draw[shift = {(-0.4, -1)}, thick] (0, 0) -- (2.2, 2.2)
																				(0, 0) -- (0, -2.2)
																				(0, 0) -- (-2.2, 2.2);

\path[shift = {(5, -1)}, fill = gray1] (2.1, -1.05) -- (0, 0) -- (0, 2.1) -- (2.1, 2.1) -- cycle;
\path[shift = {(5, -1)}, fill = gray8] (0, 2.1) -- (0, 0) -- (-2.1, 0) -- (-2.1, 2.1) -- cycle;
\path[shift = {(5, -1)}, fill = gray3] (-2.1, 0) -- (0, 0) -- (2.1, -1.05) -- (2.1, -2.1) -- (-2.1, -2.1) -- cycle;

\draw[shift = {(5, -1)}, thick] (0, 0) -- (2.2, -1.1)
																				(0, 0) -- (0, 2.2)
																				(0, 0) -- (-2.2, 0);

\path[shift = {(10.4, -1)}, fill = gray1] (2.1, 0) -- (0, 0) -- (-2.1, 1.05) -- (-2.1, 2.1) -- (2.1, 2.1) -- cycle;
\path[shift = {(10.4, -1)}, fill = gray3] (-2.1, 1.05) -- (0, 0) -- (0, -2.1) -- (-2.1, -2.1) -- cycle;
\path[shift = {(10.4, -1)}, fill = gray6] (0, -2.1) -- (0, 0) -- (2.1, 0) -- (2.1, -2.1) -- cycle;
																				
\draw[shift = {(10.4, -1)}, thick] (0, 0) -- (-2.2, 1.1)
																				(0, 0) -- (2.2, 0)
																				(0, 0) -- (0, -2.2);

\foreach \x in {-1, ..., 1} { \foreach \y in {-1, ..., 1} { 
	\draw[fill = black, shift = {(-0.4, -1)}] (\x * 2 , \y * 2 ) circle(.7mm);  
	\draw[fill = black, shift = {(5, -1)}] (2 * \x , 2 * \y ) circle(.7mm); 
	\draw[fill = black, shift = {(10.4, -1)}] (2 * \x , 2 *  \y ) circle(.7mm); 
	}}

\foreach \x in {-1, ..., 0} { \foreach \y in {-1, ..., 0} { 	
	\draw[fill = black, shift = {(-0.4, -1)}] (\x * 2  + 1, \y * 2  + 1) circle(.4mm);  
	}}
	
\foreach \x in {-1, ..., 0} { \foreach \y in {-1, ..., 0} { 	
	\foreach \w in {0, 1, 2}{ \foreach \z in {0, 1, 2, 3} {
	\draw[fill = black, shift = {(5, -1)}] (2 * \x  + \w, 2 * \y + 2 * \z / 3 ) circle(.4mm); }}
	}}	

\foreach \x in {-1, ..., 1} { \foreach \y in {-1, ..., 0} { 	
	\foreach \z in {1, ..., 12} {\draw[fill = black, shift = {(10.4, -1)}] (2 * \x , 2 * \y + \z / 6 ) circle(.4mm);}  
	}}

\node (topnode) at (5 + 0.5, 5 + 2.5) [circle,draw,thick, minimum size=0.6cm] {};
\node (midnode) at (5 + 0.5, 5 + 1.334) [circle,draw,thick, minimum size=0.6cm] {};
\node (botnode) at (5 + 2  , 5 + 0.58334) [circle,draw,thick, minimum size=0.6cm] {};
												
\draw[ultra thick, ->] (topnode)  to [bend right = 45] (-0.4, 1.2);
\draw[ultra thick, ->] (midnode)  to [bend right = 45] (5, 1.2);
\draw[ultra thick, ->] (botnode)  to [bend left = 45] (10.4,1.2);

\draw (-0.4, -3.6) node[] {$F_{(1/2, 5/2)}$};
\draw (5, -3.6) node[] {$F_{(1/2, 4/3)}$};
\draw (10.4, -3.6) node[] {$F_{(2, 7/12)}$};
	
\end{tikzpicture}
\caption{Components of a toric bouquet}
\label{fig_polycomp}
\end{figure}

Given a divisorial fan $\mathcal{S}$, its {\it tailfan} $\textrm{tail}(\mathcal{S})$ is the fan consisting of the tailcones of the $p$-divisors comprising $\mathcal{S}$. Because the coefficients $\mathcal{D}_y$ of a $p$-divisor $\mathcal{D}$ differ from its tailcone for only finitely many $y$, the polyhedral subdivisions $\mathcal{S}_y$ differ from $\textrm{tail}(\mathcal{S})$ for only finitely many $y$. By Proposition \ref{prop_irrcpt}, the fiber of $\pi$ over $y \in Y$ is equal to $TV(\textnormal{tail}(\mathcal{S}))$ for all but finitely many $y$ and specializes to a (possibly non-reduced) union of toric varieties at these finitely many points.

By the discussion above, the familiar orbit-cone correspondence for toric varieties translates into a correspondence between $T$-orbits in $\widetilde{TV}(\mathcal{S})$ and pairs $(y, F)$ where $y \in Y$ and $F \in \mathcal{S}_y$. To understand $TV(\mathcal{S})$, we will describe how the map $p$ identifies certain of these orbits in different fibers. We first consider the case of an affine $T$-variety. For a $p$-divisor $\mathcal{D}$ with tailcone $\sigma$ and a $u \in \sigma^{\vee} \cap M$, the semiample divisor $\mathcal{D}(u)$ defines a map 
\begin{align*}
\xi_u: \textrm{Loc}(\mathcal{D}) &\rightarrow \textnormal{Proj}\left( \bigoplus_{k \geq 0} \Gamma(\textrm{Loc}(\mathcal{D}), \mathcal{D}(ku)) \right).
\end{align*}

\begin{theorem}\label{thm_contraction} (\cite{poly_divisors}, Theorem 10.1)
The map $p: \widetilde{TV}(\mathcal{D}) \rightarrow TV(\mathcal{D})$ induces a surjection
\[
\{(y, F): y \in Y, F \ \textrm{is a face of} \  \mathcal{D}_y \} \rightarrow \{T-\textnormal{orbits in TV(}\mathcal{D}\textnormal{)}\}
\]
that identifies the orbits corresponding to $(y, F)$ and $(y', F')$ iff $\lambda_{F} = \lambda_{F'} \subseteq M_{\mathbb{Q}}$ and $\xi_u(y) = \xi_u(y')$ for some (equivalently, for any) $u \in \textnormal{relint}(\lambda_F)$.
\end{theorem}

In the non-affine case, the gluing maps among $\{TV(\mathcal{D})\}_{\mathcal{D} \in \mathcal{S}}$ are prescribed by the face relations between the $p$-divisors, which identifies precisely those $T$-orbits in $TV(\mathcal{D})$ and $TV(\mathcal{D}')$ corresponding to the faces $\{(y, \mathcal{D}_y \cap \mathcal{D}'_y)\}_{y \in Y}$.

\section{$T$-invariant Curves and Intersection Theory}\label{sec_curves}

In this section, we study the intersection theory of complete complexity-one $T$-varieties over a projective curve $Y$. {\bf For the rest of the paper, {\it all} $T$-varieties are complete, complexity-one $T$-varieties over a projective curve $Y$.} The ``completeness'' condition translates into the combinatorial requirement that $|\mathcal{S}_y| = N_{\mathbb{Q}}$ for all $y$. Motivated by the correspondence between $T$-invariant Cartier divisors and Cartier support functions introduced in \cite{tid}, we define the notion of a $\mathbb{Q}$-Cartier support function to encode $\mathbb{Q}$-Cartier torus invariant divisors. We will describe two kinds of $T$-invariant curves -- {\it vertical} curves and {\it horizontal} curves -- then give formulas that compute the intersection of these curves with a $T$-invariant $\mathbb{Q}$-Cartier divisor. 

\begin{definition}
Given a nontrivial $\Delta \in \textnormal{Pol}^+_{\mathbb{Q}}(N, \sigma)$ and an affine $\varphi: \Delta \rightarrow \mathbb{Q}$, the {\it linear part} of $\varphi$ is the function
\begin{align*}
\textnormal{lin}{\varphi}: \sigma &\rightarrow \mathbb{Q}\\
n &\mapsto \varphi(p + n) - \varphi(p)
\end{align*}
where $p$ is any point in $\Delta$. If $\langle \sigma \rangle \subseteq N_{\mathbb{Q}}$ is the subspace spanned by $\sigma$, the function $\textnormal{lin}\varphi$ extends uniquely to a linear function $\langle \sigma \rangle \rightarrow \mathbb{Q}$, which will also be written $\textnormal{lin}\varphi$ without risk of confusion.
\end{definition}
\begin{definition} Let $\mathcal{S}$ be the divisorial fan of a complexity-one $T$-variety over $Y$. A $\mathbb{Q}$-Cartier support function is a collection of affine functions
\[
\{h_{\mathcal{D}, y}: |\mathcal{D}_y| \rightarrow \mathbb{Q}\}_{\substack{\mathcal{D} \in \mathcal{S} \\ {y \in Y}}}
\]
with rational slope and rational translation such that 
\begin{enumerate}
\item For a fixed $y \in Y$, the functions $\{h_{\mathcal{D}, y}\}_{\mathcal{D} \in \mathcal{S}}$ define a continuous function $h_{y}: |\mathcal{S}_y| \rightarrow \mathbb{Q}$. That is, $h_{\mathcal{D}, y}$ and $h_{\mathcal{D}', y}$ agree on $\mathcal{D}_y \cap \mathcal{D}'_y$ for $\mathcal{D}, \mathcal{D}' \in \mathcal{S}$.
\item For each $\mathcal{D} \in \mathcal{S}$ with complete locus, there exists $u \in M, f \in K(Y)$ and $N \in \mathbb{Z}_{> 0}$ such that $Nh_{\mathcal{D}, y}(v) = -\textnormal{ord}_y(f) - \langle u, v \rangle$ for all $y \in Y$ and all $v \in N_{\mathbb{Q}}$.
\item If $\mathcal{D}_y, \mathcal{D}_{y'}'$ have the same tailcone, then $\textnormal{lin}{h}_{\mathcal{D}, y} = \textrm{lin}{h}_{\mathcal{D}', y'}$. 
\item For a fixed $\mathcal{D}$, $h_{\mathcal{D},y}$ differs from $\textrm{lin}{h}_{\mathcal{D}, y}$ for only finitely many $y$.
\end{enumerate}
A $\mathbb{Q}$-Cartier support function is called a {\it Cartier support function} if each $h_{\mathcal{D}, y}$ has integral slope and integral translation and $N = 1$ in condition (2). We write $CaSF(\mathcal{S})$ and $\mathbb{Q}CaSF(\mathcal{S})$ to denote the abelian group (under standard addition of functions) of Cartier support functions and $\mathbb{Q}$-Cartier support functions respectively.
\end{definition}
For any $T$-invariant Cartier divisor $D$ on $TV(\mathcal{S})$ and any $p$-divisor $\mathcal{D} \in \mathcal{S}$, we can always find an open cover $\{U_i\}$ of $Y$ for which there exists Cartier data for $\restr{D}{TV(\mathcal{D})}$ of the form $(TV(\restr{\mathcal{D}}{U_i}), f_i\chi^{u_i})$ (see proof of \cite{tid}, Prop 3.10 for details). These Cartier data define functions
\[
\{h_{\mathcal{D}, y}(v) = -\textnormal{ord}_y(f_i) - \langle u_i, v\rangle \}_{y \in U_i}
\]
which agree on the overlaps of the $U_i$ to define $h_{\mathcal{D}, y}$ for all $y$. In this way, we can define a Cartier support function for any Cartier divisor on $TV(\mathcal{S})$.

\begin{proposition}\label{prop_suppt_fxn_isomorphism}(\cite{tid}, Prop 3.10) Let $T-CaDiv(\mathcal{S})$ denote the group of $T$-invariant Cartier divisors on $TV(\mathcal{S})$. This association of a Cartier support function to a $T$-invariant Cartier divisor defines an isomorphism of groups
\[
T-CaDiv(\mathcal{S}) \cong CaSF(\mathcal{S})
\]
\end{proposition}

If $\{h_{\mathcal{D}, y}\}$ is the Cartier support function for $ND$, where $N > 0$ and $D$ is a $T$-invariant $\mathbb{Q}$-Cartier divisor, then $\{N^{-1}h_{\mathcal{D}, y}\}$ is a $\mathbb{Q}$-Cartier support function. In this way, we can associate a $\mathbb{Q}$-Cartier support function to any $T$-invariant $\mathbb{Q}$-Cartier divisor on $TV(\mathcal{S})$. The following is an immediate corollary of Proposition \ref{prop_suppt_fxn_isomorphism}.
\begin{corollary}Let $T-\mathbb{Q}CaDiv(\mathcal{S})$ denote the group of $T$-invariant $\mathbb{Q}$-Cartier divisors on $TV(\mathcal{S})$. Then the association described above is an isomorphism of groups
\[
T-\mathbb{Q}CaDiv(\mathcal{S}) \cong \mathbb{Q}CaSF(\mathcal{S})
\]
\end{corollary}

\subsection{Vertical Curves}\label{section_vertical_curves}
Toward the goal of describing the intersection theory of a $T$-variety, we study its $T$-invariant curves. We start with {\it vertical curves}, which are images (under $p$) of a $T$-invariant curve contained in a single fiber of $\pi$.

Recall from Proposition \ref{prop_irrcpt} that for $y \in Y$, the reduced fiber $\pi^{-1}(y)$ has as its irreducible components a toric variety for each vertex $v$ of $\mathcal{S}_y$. A toric variety has a $T$-invariant curve corresponding to each wall\footnote{A {\it wall} of a fan is a codimension-one cone that can be realized as the intersection of two top-dimensional cones.} of its fan (by taking the closure of the corresponding torus orbit). Translating this fact into the context of toric bouquets, we call a codimension-one element of a polyhedral complex a {\it wall} if it can be realized as the intersection of two top-dimensional polyhedra; there is a $T$-invariant curve in a toric bouquet for each wall of the corresponding polyhedral complex. In this section, we study the curves in $\widetilde{TV}(\mathcal{S})$ and $TV(\mathcal{S})$ corresponding to these $T$-invariant curves.

Fix a $T$-variety $TV(\mathcal{S})$ and a point $y \in Y$. Let $\tau \in \mathcal{S}_y$ be a wall of the polyhedral complex $\mathcal{S}_{y}$, let $\mathcal{D}, \mathcal{D}' \in \mathcal{S}$ be two $p$-divisors for which $\tau = \mathcal{D}_{y} \cap \mathcal{D}_{y}'$, let $\lambda_{\tau, \mathcal{D}} \subseteq M_{\mathbb{Q}}$ be the cone in $\mathcal{N}(\mathcal{D}_y)$ dual to $\tau$, and let $u_{\tau, \mathcal{D}}$ be the semigroup generator of $M_{\lambda_{\tau, \mathcal{D}}}$. As usual, unweildy notation  obfuscates a simple picture: if  $\mathcal{D}$ and $\mathcal{D}'$ have polyhedral coefficients over $y$ as shown in Figure \ref{fig_vertcurve} ($\tau$ is the horizontal plane in a single orthant at a height of $2/3$), then the sublattice $\mathbb{Z} \cdot u_{\tau, \mathcal{D}} = M_{\lambda_{\tau, \mathcal{D}}} \cup M_{\lambda_{\tau, \mathcal{D}'}}$ consists of the bold elements of the vertical axis of  $M \cong \mathbb{Z}^3$ shown on the right.

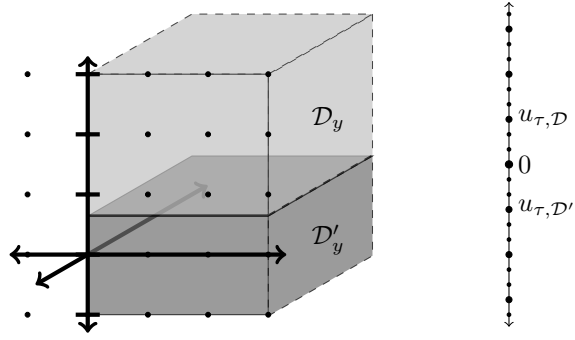
\begin{figure}[!ht]
\centering
\begin{tikzpicture}[scale = 0.8]
\draw[ultra thick, ->] (0, 0) -- ++(30:2.3);

\fill[fill=gray4, draw=black, opacity = 0.7] (0,0.64) -- (3, 0.64) -- (3, -1) --(0, -1) --(0,0.64);
\fill[fill=gray4, draw=black, dashed, opacity = 0.7] (3, 0.64) -- ++(30:2) -- ++(-90: 1.66) -- ++(210:2) -- (3, 0.64);
\fill[fill=gray2, draw=black, opacity = 0.7] (0,0.64) -- (3, 0.64)  -- ++(30:2)  -- ++(180:3)  --(0,0.64); 

\fill[fill=gray8, draw=black, dashed, opacity = 0.7] (3, 0.66) -- ++(30:2) -- ++(90:2.33) -- ++(210:2) -- (3, 0.66);
\fill[fill=gray8, draw=black, dashed, opacity = 0.7] (3, 3) -- ++(30:2) -- ++(180:3) -- ++(210:2) -- (3, 3);
\fill[fill=gray8, draw=black, opacity = 0.7] (0,0.66) -- (3, 0.66) -- (3, 3) --(0, 3) --(0,0.66);

\draw[ultra thick, <->] (0, -1.3) -- (0, 3.3);
\foreach \i in {-1, 0, 1, 2, 3}
	{ \draw[ultra thick] (-0.2, \i) -- (0.2, \i);};

\draw[ultra thick, <->] (-1.3, 0) -- (3.3, 0);
\draw[ultra thick, ->] (0, 0) -- ++(210:1);

\foreach \x in {-1, ..., 3} { \foreach \y in {-1, ..., 3} {
	\draw[fill = black] (\x , \y ) circle (0.4mm);
	}} 

\draw[shift = {(7, -1)}, <->] (0, -0.2) -- (0, 5.2);

\draw (4, 2.25) node[] {$\mathcal{D}_y$};
\draw (4, 0.25) node[] {$\mathcal{D}'_y$};

\foreach \y in {0, ..., 20}
	{
	\draw[shift = {(7, -1)}, fill = black] (0, \y / 4) circle (.3mm);
	} 
	\draw[shift = {(7, -1)}, fill = black] (0, 4.75) circle (.5mm);
	\draw[shift = {(7, -1)}, fill = black] (0, 4) circle (.5mm);
	\draw[shift = {(7, -1)}, fill = black] (0, 3.25) circle (.5mm) node[right] {$u_{\tau, \mathcal{D}}$};
	\draw[shift = {(7, -1)}, fill = black] (0, 2.5) circle (.6mm) node[right] {$0$};
	\draw[shift = {(7, -1)}, fill = black] (0, 1.75) circle (.5mm) node[right] {$u_{\tau, \mathcal{D}'}$};;
	\draw[shift = {(7, -1)}, fill = black] (0, 1) circle (.5mm);
	\draw[shift = {(7, -1)}, fill = black] (0, .25) circle (.5mm);
	
\end{tikzpicture}
\caption{The sublattice $\mathbb{Z} \cdot u_{\tau, \mathcal{D}}$ corresponding to the wall $\tau = \mathcal{D}_y \cap \mathcal{D}'_y$}
\label{fig_vertcurve}
\end{figure}

If $g \in K(Y)$ is Cartier data for $\mathcal{D}(u_{\tau, \mathcal{D}})$ in some neighborhood of $y$, the maps 
\begin{align*}
\Gamma(\textnormal{Loc}(\mathcal{D}), \mathcal{O}(\mathcal{D})) &\rightarrow \mathbb{C}[z]\\
f\chi^u &\mapsto \left\{ \begin{array}{c c} 0 & u \notin M_{\lambda_{\tau}, \mathcal{D}}\\ (g^kf)(y)z^k & u = ku_{\tau, \mathcal{D}} \end{array} \right.
\end{align*}
\begin{align*}
\Gamma(\textnormal{Loc}(\mathcal{D}'), \mathcal{O}(\mathcal{D}')) &\rightarrow \mathbb{C}[z^{-1}]\\
f\chi^u &\mapsto \left\{ \begin{array}{c c} 0 & u \notin M_{\lambda_{\tau}, \mathcal{D}'}\\ (g^{-k}f)(y)z^{-k} & u = -ku_{\tau, \mathcal{D}} \end{array} \right.
\end{align*}
glue together to induce a map
\begin{equation}\label{eqn_vertparam}
\mathbb{P}^1 \rightarrow {TV}(\mathcal{S})
\end{equation}
the image of which we will call the {\it vertical curve} $C_{\tau, y}$. 
\begin{proposition}\label{prop_vc}
The vertical curve $C_{\tau, y}$ is the image under $p$ of the closure of the torus orbit in $TB(\mathcal{S}_y) \subseteq \widetilde{TV}(\mathcal{S})$ corresponding to the wall $\tau$.
\end{proposition}
\begin{proof}
For any affine open $U \subseteq Y$ containing $y$, Map \ref{eqn_vertparam} factors
\begin{equation}\label{eqn_factored}
\mathbb{P}^1 \rightarrow \pi^{-1}(U) \subseteq \widetilde{TV}(\mathcal{S}) \overset{p}{\rightarrow} TV(\mathcal{S})
\end{equation}

where $\mathbb{P}^1 \rightarrow \pi^{-1}(U)$ is given by
\begin{align}\label{eqn_abcd}
\Gamma(U, \mathcal{O}(\mathcal{D})) &\rightarrow \mathbb{C}[z]\\
f\chi^u &\mapsto \left\{ \begin{array}{c c} 0 & u \notin M_{\lambda_{\tau}, \mathcal{D}}\\ (g^kf)(y)z^k & u = ku_{\tau, \mathcal{D}} \end{array} \right.\notag
\end{align}
\begin{align*}
\Gamma(U, \mathcal{O}(\mathcal{D}')) &\rightarrow \mathbb{C}[z^{-1}]\\
f\chi^u &\mapsto \left\{ \begin{array}{c c} 0 & u \notin M_{\lambda_{\tau}, \mathcal{D}'}\\ (g^{-k}f)(y)z^{-k} & u = -ku_{\tau, \mathcal{D}} \end{array} \right.
\end{align*}
Therefore, it suffices to show that the image of $\mathbb{P}^1 \rightarrow \pi^{-1}(U)$ is the closure of the torus orbit in $TB(\mathcal{S}_y) \subseteq \widetilde{TV}(\mathcal{S})$ corresponding to the wall $\tau$. To do so, we recall some relevant details about the isomorphism between the reduced fibers of $\pi$ and a dappled toric bouquet (see \cite{poly_divisors}, Proposition 7.10 for details). This isomorphism is constructed by first choosing a collection of functions $\{g_{\mathcal{D}(u)} \in K(Y)\}_{u \in S_{\Delta}}$ such that, after possibly shrinking $U$,
\[
\restr{\textnormal{div}(g_{\mathcal{D}(u)})}{U} = \restr{\mathcal{D}(u)}{U} \hspace{1cm} \textnormal{and} \hspace{1cm}  g_{\mathcal{D}(u + u')} = g_{\mathcal{D}(u)}g_{\mathcal{D}(u')}
\]
Then the isomorphism between the fiber and the dappled toric bouquet is induced by
\begin{align} \label{eqn_efgh}
\Gamma(U, \mathcal{O}(\mathcal{D})) &\rightarrow \mathbb{C}[\mathcal{N}(\mathcal{D}_y), S_{\mathcal{D}_y}]\\
f\chi^u &\mapsto \left\{\begin{array}{c c}(g_{\mathcal{D}(u)}f)(y)\chi^u &\textnormal{ if $u \in S_{\mathcal{D}_y}$}\\ 0 &\textnormal{ otherwise} \end{array} \right. \notag
\end{align}
(and similarly for $\mathcal{D}'$). On the other hand, the closure of the torus orbit corresponding to $\tau$ in the toric bouquet is parametrized by gluing the maps
\begin{align}\label{eqn_ijkl}
\mathbb{C}[\mathcal{N}(\mathcal{D}_y), S_{\mathcal{D}_y}] &\rightarrow \mathbb{C}[z]\\\notag
\chi^u &\mapsto \left\{\begin{array}{c c}z^k &\textnormal{ if $u = ku_{\tau, \mathcal{D}}, k \in \mathbb{Z}$}\\ 0 &\textnormal{ otherwise} \end{array} \right.\\\notag
\mathbb{C}[\mathcal{N}(\mathcal{D}_y'), S_{\mathcal{D}_y'}] &\rightarrow \mathbb{C}[z^{-1}]\\\notag
\chi^u &\mapsto \left\{\begin{array}{c c}z^{-k} &\textnormal{ if $u = -ku_{\tau, \mathcal{D}}, k \in \mathbb{Z}$}\\ 0 &\textnormal{ otherwise} \end{array} \right.
\end{align}
The composition of Equation \ref{eqn_efgh} and Equation \ref{eqn_ijkl} yields Equation \ref{eqn_abcd}, proving the claim.
\end{proof}
To find a formula that calculates the intersection between $C_{\tau, y}$ and a $T$-invariant Cartier divisor $D$, we pick Cartier data for $D$ that includes two sets of the form
\[
\{(TV(\restr{\mathcal{D}}{U}), f\chi^u), (TV(\restr{\mathcal{D}'}{U'}), f' \chi^{u'})\}
\]
where $U, U' \subseteq Y$ are open sets containing $y$. The Cartier support function for $D$ includes
\[
h_{\mathcal{D}, y} = -\textnormal{ord}_{y}(f) - \langle u, v\rangle \hspace{1cm} \textrm{and} \hspace{1cm} h_{\mathcal{D}', y} = -\textnormal{ord}_{y}(f') - \langle u', v\rangle.
\]
Because $h_{\mathcal{D}, y}$ and $h_{\mathcal{D}', y}$ agree on $\tau$, it must be the case that
\[
\textnormal{ord}_y(f) - \textnormal{ord}_y(f') + \langle u - u', v\rangle = 0
\]
for all $v \in \tau$. In particular, $u - u' \in \mathbb{Q} \cdot u_{\tau, \mathcal{D}}$. Moreover, since $ \langle u - u', v\rangle = \textnormal{ord}_{y}(f') - \textnormal{ord}_{y}(f) \in \mathbb{Z}$, it must be the case that $\langle u - u', v\rangle \in \mathbb{Z}$ for $v \in \tau$. Therefore, $u - u' = ku_{\tau, \mathcal{D}}$ for some $k \in \mathbb{Z}$, and the quotient of the two Cartier data is $f\chi^u / f'\chi^{u'} = (f/f') \chi^{ku_{\tau, \mathcal{D}}}$. Under the parametrization of $C_{\tau, y}$ in Equation \ref{eqn_vertparam}, this rational function pulls back to $(g^kf/f')(y) z^k$ on $C_{\tau, y} \cong \mathbb{P}^1$, where $g$ is Cartier data for $\mathcal{D}(u_{\tau, \mathcal{D}})$. Therefore, the degree of the pullback of $D$ onto $C_{\tau, y}$ is $k$. This is precisely $\mu_{\tau}^{-1}\langle u - u', n_{\tau, \mathcal{D}} \rangle$, where $\mu_{\tau}$ is the index of $\mathbb{Z} \cdot u_{\tau, \mathcal{D}}$ in $\mathbb{Q} \cdot u_{\tau, \mathcal{D}} \cap M$ and $n_{\tau, \mathcal{D}} \in N$ is any representative of the generator of $N / (u_{\tau, \mathcal{D}})^{\perp}$ that pairs positively with $u_{\tau, \mathcal{D}}$ (equivalently, $n_{\tau, \mathcal{D}}$ is any element of $N$ such that $\langle n_{\tau, \mathcal{D}}, u_{\tau, \mathcal{D}}\rangle = \mu_{\tau}$).
\[
\langle D, C_{\tau, y} \rangle = \mu_{\tau}^{-1}\langle u - u', n_{\tau, \mathcal{D}} \rangle
\]
or, using the language of Cartier support functions, 
\begin{equation}\label{equation_vertintersect}
\langle D, C_{\tau, y} \rangle = \mu_{\tau}^{-1}(\textrm{lin}h_{\mathcal{D}', y} - \textrm{lin}h_{\mathcal{D}, y})(n_{\tau, \mathcal{D}})
\end{equation}
By linearity, the same formula applies when $D$ is a $T$-invariant $\mathbb{Q}$-Cartier divisor.

\begin{example} Let $\mathcal{D}, \mathcal{D}' \in \mathcal{S}$ be $p$-divisors that have the slices shown in Figure \ref{fig_vertcurve}. Suppose that with respect to the standard basis given by the coordinate axes in the picture, a $T$-invariant divisor $D$ has the following Cartier support functions
\[
h_{\mathcal{D}, y}(v) = -10 + \langle (9, 4, 17), v\rangle \hspace{0.5cm} \textnormal{and} \hspace{0.5cm} h_{\mathcal{D}', y}(v) = 0 + \langle (9, 4, 2), v\rangle
\]
Then 
\[
\langle D, C_{\tau, y} \rangle = 3^{-1}\langle(0, 0, -15), (0, 0, 1)\rangle = -5
\]
\end{example}

\subsection{Horizontal Curves} Let $\sigma$ be a full-dimensional cone of $\textnormal{tail}(\mathcal{S})$. Because the $T$-varieties we study are complete, every $\mathcal{S}_y$ contains a polyhedron with tailcone $\sigma$. Such a polyhedron corresponds to a fixed point in the fiber $\pi^{-1}(y)$. Taking the union (as $y$ varies) of these fixed points defines a curve $\widetilde{C}_{\sigma} \subseteq \widetilde{TV}(\mathcal{S})$. Theorem \ref{thm_contraction} shows that $p$ contracts $\widetilde{C}_{\sigma}$ precisely if there is some $\mathcal{D} \in \mathcal{S}$ with tailcone $\sigma$ and complete locus. In this case, we say that $\sigma$ is {\it marked}.
\begin{definition}
A cone $\sigma$ of $\textnormal{tail}(\mathcal{S})$ is {\it marked} if $\sigma$ is the tailcone of a $p$-divisor $\mathcal{D} \in \mathcal{S}$ with complete locus.
\end{definition}
When $\sigma$ is unmarked, Theorem \ref{thm_contraction} shows that no distinct points of $\widetilde{C}_{\sigma}$ are identified by $p$. Toward the goal of finding an intersection formula for these {\it horizontal curves} $C_{\sigma} := p(\widetilde{C}_{\sigma})$, we parametrize them. Let $TV(\mathcal{S})$ be a $T$-variety and let $\sigma$ be an unmarked full-dimensional cone of $\textrm{tail}(\mathcal{S})$. For $\mathcal{D} \in \mathcal{S}$ with tailcone $\sigma$, we have a map of rings
\begin{align} \label{map_hcurve}
\varphi_{\mathcal{D}}: \Gamma(TV(\mathcal{D}), \mathcal{O}(\mathcal{D})) &\rightarrow \Gamma(\textnormal{Loc}(\mathcal{D}), \mathcal{O}_{Y})\\
f\chi^u &\mapsto \left\{\begin{array}{c c} f & \textnormal{if} \ u = 0\\ 0 & \textnormal{otherwise} \end{array}\right. \notag
\end{align}
Because each $\{\textrm{Loc}(\mathcal{D}) \mid \textnormal{tail}(\mathcal{D}) = \sigma\}$ is affine, these glue into a map
\[
s_{\sigma}: Y \hookrightarrow TV(\mathcal{S})
\]
where we used the fact that $\mathcal{S}$ is complete (so $|\mathcal{S}_y| = N_{\mathbb{Q}}$ for all $y$) to deduce that $Y$ is covered by $\{\textnormal{Loc}(\mathcal{D}) \mid \textnormal{tail}(\mathcal{D}) = \sigma\}$. The map $s_{\sigma}$ factors through $\widetilde{TV}(\mathcal{S})$. By carefully following the isomorphism between the fibers of $\pi$ and the corresponding toric bouquets (as in the proof of Proposition \ref{prop_vc}), we see that the image of $s_{\sigma}$ indeed equals the horizontal curve $p(\widetilde{C}_{\sigma})$. 

We can use this parametrization to find an intersection formula for $T$-invariant divisors and horizontal curves. Fix a cone $\sigma$ of $\textnormal{tail}(\mathcal{S})$ of full dimension and a $T$-invariant Cartier divisor $D$ with Cartier support function $\{h_{\mathcal{D}, y}\}$. Because $\sigma$ has full dimension, there is a unique $u_{\sigma} \in M$ and collection of integers $\{a_{y} \in \mathbb{Z}\}_{y \in Y}$ such that for each $\mathcal{D}$ with tailcone $\sigma$ and each $y \in \textrm{Loc}(\mathcal{D})$,
\[
h_{y, \mathcal{D}}(v) = -a_y - \langle u_{\sigma}, v\rangle.
\]
We can find Cartier data for $D$ whose open sets and rational functions are of the form
\[
(TV(\restr{\mathcal{D}}{U}), f_{\mathcal{D}, U} \chi^{u_{\sigma}})
\]
for open sets $U \subseteq Y$. Then $\textnormal{ord}_y(f_{\mathcal{D}, U}) = a_y$ for all $\mathcal{D}$ with tailcone $\sigma$ and $y \in U$. When $\sigma$ is unmarked, the open sets $U$ appearing in the Carter data are affine, and the pullback of the transition function $f_{\mathcal{D}, U}f^{-1}_{\mathcal{D}', U'}\chi^0$ onto the curve $C_{\sigma} \cong Y$ is the function $f_{\mathcal{D}, U}f^{-1}_{\mathcal{D}', U'}$ on $U \cap U'$. That is, the functions $f_{\mathcal{D}, U}$ appearing in the Cartier data for $D$ are themselves the Cartier data for the pullback of $D$ onto $C_{\sigma} \cong Y$. As a Weil divisor, the pullback of $D$ onto $C_{\sigma}$ is $\sum a_y[y]$; we call this divisor $D_{\sigma}$.

\begin{definition}
Given a $\mathbb{Q}$-Cartier support function $\{h_{\mathcal{D}, y}\}$, a cone $\sigma$ of full dimension in $\textnormal{tail}(\mathcal{S})$, and a point $y$, there is a unique $a_y \in \mathbb{Z}$ such that for every $\mathcal{D}$ with tailcone $\sigma$ and $\textnormal{Loc}(\mathcal{D}) \ni y$,
\[
h_{\mathcal{D}, y} = -a_y - \textnormal{lin}(h_{\mathcal{D}, y}).
\]
Then define
\[
D_{\sigma} = \sum_{y \in Y} a_y[y]
\]
\begin{remark}\label{remark_horizdsig}
The definition of $D_{\sigma}$ makes sense even when $\sigma$ is marked. However, if $\mathcal{D}$ has complete locus, then by (\cite{tid}, Proposition 3.1) every invariant Cartier divisor on $TV(\mathcal{D})$ is principal. It follows that $\textnormal{deg}(D_{\sigma}) = 0$ for every marked $\sigma$.
\end{remark}
\end{definition}
\begin{remark} Compare this definition to (\cite{tid}, Definition 3.26). In our notation, $D_\sigma = -\restr{h}{\sigma}(0)$.
\end{remark}

With this new definition, we can summarize the discussion above with the following equation for the intersection theory of a T-invariant divisor with a horizontal curve.
\[
\langle D, C_{\sigma} \rangle = \textnormal{deg}(D_{\sigma})
\]
By linearity, the same formula applies when $D$ is a $T$-invariant $\mathbb{Q}$-Cartier divisor.

\section{The $T$ Cone Theorem}\label{sec_tct}
Given a normal variety $X$, let $Z_1(X)$ be the proper 1-cycles, and define
\[
N^1(X) := (\textnormal{CaDiv}(X)/\sim) \otimes_{\mathbb{Z}} \mathbb{R} \hspace{1cm} N_1(X) := (Z_1(X)/\sim) \otimes_{\mathbb{Z}} \mathbb{R}
\]
where $\sim$ denotes numerical equivalence of divisors in the first definition, and numerical equivalence of curves in the second. The vector space $N^1(X)$ contains the cone $\textnormal{Nef}(X)$ generated by classes of nef divisors, and the vector space $N_1(X)$ contains the cone $NE(X)$ generated by classes of irreducible complete curves. The {\it Mori cone} $\overline{NE}(X)$ is the closure of $NE(X)$. With respect to the intersection product, $N_1(X)$ and $N^1(X)$ are dual vector spaces, and the cones $\textnormal{Nef}(X), \overline{NE}(X)$ are dual cones.

When $X$ is the toric variety of a fan $\Sigma$, the closure of the torus orbit corresponding to a wall of $\Sigma$ defines an element of $\overline{NE}(X)$. The celebrated toric cone theorem (\cite{cox}, Theorem 6.3.20(b)) states that $\overline{NE}(X)$ is generated as a cone by these classes. In this section, we prove the corresponding result for $T$-varieties. We continue to assume that all $T$-varieties are complete complexity-one $T$-varieties over a projective curve $Y$.
\begin{theorem}\label{tct} Let $TV(\mathcal{S})$ be an $n$-dimensional $T$-variety, and let $y' \in Y$ be any point for which $\mathcal{S}_{y'} = \textnormal{tail}(\mathcal{S})$. Then 
\begin{equation}\label{tct_sum}
\overline{NE}(TV(\mathcal{S})) = \sum_{\substack{y \in Y\\ \tau \ \textnormal{a wall} \textnormal{ of } \mathcal{S}_y \\ \textnormal{dim}(\textnormal{tail}(\tau)) < n-1}} \mathbb{R}_{\geq 0} [C_{\tau, y}]  + \sum_{\substack{\tau \ \textnormal{a wall} \\ \textnormal{of } \textnormal{tail}(\mathcal{S})}} \mathbb{R}_{\geq 0} [C_{\tau, y'}] + \sum_{\substack{\sigma \in \textnormal{tail}(\mathcal{S}) \\ \textnormal{dim}(\sigma) = n-1 \\ \sigma \textnormal{ unmarked}}} \mathbb{R}_{\geq 0} [C_{\sigma}]
\end{equation}
\end{theorem}
For the proof, we review two important facts about divisors on $T$-varieties. 
\begin{proposition} Any Cartier divisor $D$ on a $T$-variety $TV(\mathcal{S})$ is linearly equivalent to a $T$-invariant Cartier divisor.
\end{proposition}
Different authors have different definitions of concavity; to us, a function $\varphi: N_{\mathbb{Q}} \rightarrow \mathbb{Q}$ is concave if $\varphi(tv + (1-t)w) \geq t\varphi(v) + (1-t)\varphi(w)$ for all $v, w \in N_{\mathbb{Q}}$ and all $t \in [0, 1]$
\begin{proposition}\label{prop_nefconcave}(\cite{tid}, Corollary 3.29) A T-invariant Cartier divisor $D \in T-CaDiv(\mathcal{S})$ with Cartier support function $\{h_{\mathcal{D}, y}\}$ is nef iff all $h_y$ are concave and $\textnormal{deg} (D_{\sigma}) \geq 0$ for every maximal cone $\sigma$ of the tailfan.
\end{proposition}

Toward our goal of proving Theorem \ref{tct}, we will use Proposition \ref{prop_nefconcave} to show that a Cartier divisor is nef if it intersects all vertical and horizontal curves nonnegatively. The proof of this fact requires a combinatorial lemma. Given a Cartier support function $\{h_{\mathcal{D}, y}\}$ and any $\mathcal{D} \in \mathcal{S}, y \in Y$ such that $\textnormal{dim}(\mathcal{D}_y) = \textnormal{dim}(N_{\mathbb{Q}})$, define $\widetilde{h}_{\mathcal{D}, y}: N_{\mathbb{Q}} \rightarrow \mathbb{Q}$ to be the unique affine function that extends $h_{\mathcal{D}, y}: |\mathcal{D}_y| \rightarrow \mathbb{Q}$.
\begin{lemma}\label{lemma_comb} Let $\{h_{\mathcal{D}, y}\}$ be a Cartier support function. The following are equivalent
\begin{itemize}
\item $h_y: N_{\mathbb{Q}} \rightarrow \mathbb{Q}$ is concave.
\item For every wall $\tau = \mathcal{D}_y \cap \mathcal{D}_y'$ of $\mathcal{S}_y$, there is some $v \in \mathcal{D}_y' \backslash \mathcal{D}_y$ with $h_{\mathcal{D}', y}(v) \leq \widetilde{h}_{\mathcal{D}, y}(v)$.
\end{itemize}
\end{lemma}
\begin{proof} This is a straightforward extension of (\cite{cox}, Lemma 6.1.5 $(a) \iff (d)$) (where it is proved for Cartier support functions on a fan).
\end{proof}

\begin{proposition}\label{prop_nefintersect} A Cartier divisor $D \in \textnormal{CaDiv}(TV(\mathcal{S}))$ is nef iff $\langle D, C\rangle \geq 0$ for all vertical and horizontal curves $C$.
\end{proposition}
\begin{proof} The forward direction follows from the definition of nef. To prove the reverse direction, let $D \in \textrm{CaDiv}(\mathcal{S})$ satisfy the condition that $\langle D, C \rangle \geq 0$ for all vertical and horizontal curves. Replace $D$ with a linearly equivalent $T$-invariant divisor and let $\{h_{\mathcal{D}, y}\}$ be its Cartier support function. Let $\tau = \mathcal{D}_y \cap \mathcal{D}_y'$ be a wall of $\mathcal{S}_y$. Fix any $n_{\tau, \mathcal{D}} \in N$ with $\langle n_{\tau, \mathcal{D}}, u_{\tau, \mathcal{D}}\rangle = \mu_{\tau}$ and any $v_{\tau} \in \textrm{relint}(\tau)$. Then pick $\epsilon > 0$ such that $v := v_{\tau} + \epsilon n_{\tau, \mathcal{D}} \in \mathcal{D}_y \backslash \mathcal{D}_y'$. Then
\begin{align*}
h_{\mathcal{D}, y}(v) &= h_{\mathcal{D}, y}(v_{\tau}) +  \textnormal{lin}h_{\mathcal{D}, y}(\epsilon n_{\tau, \mathcal{D}})\\
\widetilde{h}_{\mathcal{D}', y}(v) &= {h}_{\mathcal{D}', y}(v_{\tau}) + \textnormal{lin}{h}_{\mathcal{D}', y}(\epsilon n_{\tau, \mathcal{D}}).
\end{align*}
Because $h_{\mathcal{D}, y}$ and $h_{\mathcal{D}', y}$ agree on $\tau$,
\begin{align*}
\widetilde{h}_{\mathcal{D}', y}(v) - h_{\mathcal{D}, y}(v) = (\textnormal{lin}h_{\mathcal{D}', y} - \textnormal{lin}{h}_{\mathcal{D}, y})(\epsilon n_{\tau, \mathcal{D}}) \geq 0
\end{align*}
where the final inequality comes from applying Equation \ref{equation_vertintersect} to the fact that $\langle D, C_{\tau, y}\rangle \geq 0$. Because this holds for all walls in all slices $\mathcal{S}_y$, we conclude by Lemma \ref{lemma_comb} that each $h_y$ is concave.

To show that $\textnormal{deg} (D_{\sigma}) \geq 0$ for every maximal cone $\sigma$ of the tailfan, observe that if $\sigma$ is marked, then $\textnormal{deg}(D_{\sigma}) = 0$ by Remark \ref{remark_horizdsig}; if $\sigma$ is unmarked, then $\textnormal{deg} (D_{\sigma}) = \langle D, C_{\sigma} \rangle \geq 0$.
\end{proof}

To put Proposition \ref{prop_nefintersect} in context, remember that a $T$-variety has infinitely many distinct vertical curves. Indeed, if $\tau$ is a wall of $\textnormal{tail}(\mathcal{S})$, then for every $y \in Y$ there is (by completeness) a vertical curve $C_{\tau', y}$ where $\tau'$ is a wall of $\mathcal{S}_y$ with tailcone $\tau$. The next proposition shows that the classes of all such curves lie on a single ray of $N_1(TV(\mathcal{S}))$.

\begin{proposition}\label{prop_linequiv} Let $\tau = \sigma \cap \sigma'$ be a wall of $\textnormal{tail}(\mathcal{S})$, where $\sigma, \sigma'$ are full dimensional cones of $\textnormal{tail}(\mathcal{S})$. The classes
\[
\mathcal{C}_\tau = \left\{[C_{\tau', y}] \middle| \begin{array}{c}\tau' = \mathcal{D}_y \cap \mathcal{D}_y' \textnormal{ for some } \mathcal{D}, \mathcal{D}' \textnormal{ with} \\ \textnormal{tail}(\mathcal{D}) = \sigma, \textnormal{tail}(\mathcal{D}') = \sigma' \end{array} \right\} \subseteq N_1(TV(\mathcal{S}))
\]
are positive multiples of each other. Specifically, $[C_{\tau_1, y_1}] = \mu_{\tau_1}^{-1}\mu_{\tau_2}[C_{\tau_2, y_2}]$.
\end{proposition}
\begin{proof} 
Let $\{h_{\mathcal{D}, y}\}$ be the Cartier support funtion of some $D \in \textrm{T-CaDiv}(TV(\mathcal{S}))$. All $h_{\mathcal{D}, y}$ with tail($\mathcal{D}) = \sigma$ (respectively $\sigma'$) will have the same linear part, say $-u_{\sigma} \in M_{\mathbb{Q}}$ (respectively $-u_{\sigma'} \in M_{\mathbb{Q}}$). Then for two classes $[C_{\tau_1, y_1}], [C_{\tau_2, y_2}] \in \mathcal{C}_{\tau}$, Equation \ref{equation_vertintersect} calculates the intersections as
\[
\langle D, C_{\tau_1, y_1} \rangle = \mu_{\tau_1}^{-1}\langle u_{\sigma} - u_{\sigma'}, n_{\tau_1, \mathcal{D}} \rangle \hspace{1cm} \langle D, C_{\tau_2, y_2} \rangle = \mu_{\tau_2}^{-1}\langle u_{\sigma} - u_{\sigma'}, n_{\tau_2, \mathcal{D}} \rangle
\]
Since we can choose $n_{\tau_1, \mathcal{D}} = n_{\tau_2, \mathcal{D}}$, it follows that $\langle D, C_{\tau_1, y_1} \rangle  = \mu_{\tau_1}^{-1}\mu_{\tau_2}\langle D, C_{\tau_2, y_2} \rangle$ for all $D$.
\end{proof}

We are finally ready to prove Theorem \ref{tct}. Using the propositions above, the proof is nearly identical to the proof of the toric cone theorem in (\cite{cox}, Theorem 6.3.20(b)).
\begin{proof}(Theorem \ref{tct}) Let $\Gamma$ be the rational polyhedral cone in $NE(TV(\mathcal{S}))$ defined by the right hand side of Equation \ref{tct_sum}. By definition, $\Gamma$ includes the classes of all horizontal curves; by Proposition \ref{prop_linequiv}, it also includes the classes of all vertical curves. Therefore, Proposition \ref{prop_nefintersect} implies that $\Gamma^{\vee} = \textrm{Nef}(TV(\mathcal{S}))$, so $\Gamma = \Gamma^{\vee \vee} = \overline{NE}(TV(\mathcal{S}))$.
\end{proof}

\section{Examples}\label{sec_examples}
\subsection{Example 1}

Consider the divisoral fan $\mathcal{S}$ shown in Figure \ref{fig_pctb}. $TV(\mathcal{S})$ is the projectivized cotangent bundle of the first Hirzebruch surface. All horizontal divisors\footnote{See \cite{tid} for a definition and description of horizontal and vertical divisors} in $\widetilde{TV}(\mathcal{S})$ are contracted. For each vertical divisor $D_{[y], v}$ and each maximal $p$-divisor $\mathcal{D}_i \in \mathcal{S}$, we write the Weil divisor $\sum a_y[y]$ and an element $u \in M$ in Table \ref{tab_pctbdivisors} to encode the Cartier support function  $\{h_{\mathcal{D}_i, y}(w) = -a_y - \langle u, w \rangle\}$ of $D_{[y], v}$. For example, the Cartier support function for $D_{[0], (0, 0)}$ includes $h_{\mathcal{D}_4, \infty}(v) = 1 - \langle (-2, -1), v\rangle$.

\begin{figure}[!ht]
\centering
\begin{tikzpicture}[scale = 0.95]

\path[shift = {(0, 0)}, fill = gray1] (1.6, 0) -- (0, 0) -- (0, 0.5) -- (1.1, 1.6) -- (1.6, 1.6) -- cycle;
\path[shift = {(0, 0)}, fill = gray3] (1.1, 1.6) -- (0, 0.5) -- (0, 1.6) -- cycle;
\path[shift = {(0, 0)}, fill = gray5] (0, 1.6) -- (0, 0.5) -- (-0.55, 1.6) -- cycle;
\path[shift = {(0, 0)}, fill = gray7] (-0.55, 1.6) -- (0, 0.5) -- (0, 0) -- (-1.6, 1.6) -- cycle;
\path[shift = {(0, 0)}, fill = gray2] (-1.6, 1.6) -- (0, 0) -- (0, -0.5) -- (-1.6, -0.5) -- cycle;
\path[shift = {(0, 0)}, fill = gray4] (-1.6, -0.5) -- (0, -0.5) -- (0, -1.6) -- (-1.6, -1.6) -- cycle;
\path[shift = {(0, 0)}, fill = gray6] (0, -1.6) -- (0, -0.5) -- (1.1, -1.6) -- cycle;
\path[shift = {(0, 0)}, fill = gray8] (1.1, -1.6) -- (0, -0.5) -- (0, 0) -- (1.6, 0) -- (1.6, -1.6) -- cycle;

\path[shift = {(4.5, 0)}, fill = gray1] (1.6, 0) -- (0.5, 0) -- (1.6, 1.1) -- cycle;
\path[shift = {(4.5, 0)}, fill = gray3] (1.6, 1.1) -- (0.5, 0) -- (0, 0) -- (0, 1.6) -- (1.6, 1.6) -- cycle;
\path[shift = {(4.5, 0)}, fill = gray5] (0, 1.6) -- (0, 0) -- (-0.85, 1.6) -- cycle;
\path[shift = {(4.5, 0)}, fill = gray7] (-0.85, 1.6) -- (0, 0) -- (-1.6, 1.6) -- cycle;
\path[shift = {(4.5, 0)}, fill = gray2] (-1.6, 1.6) -- (0, 0) -- (-1.6, 0) -- cycle;
\path[shift = {(4.5, 0)}, fill = gray4] (-1.6, 0) -- (0, 0) -- (0, -1.6) -- (-1.6, -1.6) -- cycle;
\path[shift = {(4.5, 0)}, fill = gray6] (0, -1.6) -- (0, 0) -- (0.5, 0) -- (1.6, -1.1) -- (1.6, -1.6) -- cycle;
\path[shift = {(4.5, 0)}, fill = gray8] (1.6, -1.1) -- (0.5, 0) -- (1.6, 0) -- cycle;

\path[shift = {(9, 0)}, fill = gray1] (1.6, 0) -- (0, 0) -- (1.6, 1.6) -- cycle;
\path[shift = {(9, 0)}, fill = gray3] (1.6, 1.6) -- (0, 0) -- (0, 1.6) -- cycle;
\path[shift = {(9, 0)}, fill = gray5] (0, 1.6) -- (0, 0) -- (-0.5, 0.5) -- (-1.05, 1.6) -- cycle;
\path[shift = {(9, 0)}, fill = gray7] (-1.05, 1.6) -- (-0.5, 0.5) -- (-1.6, 1.6) -- cycle;
\path[shift = {(9, 0)}, fill = gray2] (-1.6, 1.6) -- (-0.5, 0.5) -- (-1.6, 0.5) -- cycle;
\path[shift = {(9, 0)}, fill = gray4] (-1.6, 0.5) -- (-0.5, 0.5) -- (0, 0) -- (0, -1.6) -- (-1.6, -1.6) -- cycle;
\path[shift = {(9, 0)}, fill = gray6] (0, -1.6) -- (0, 0) -- (1.6, -1.6) -- cycle;
\path[shift = {(9, 0)}, fill = gray8] (1.6, -1.6) -- (0, 0) -- (1.6, 0) -- cycle;

\foreach \x in {-3, ..., 3} { \foreach \y in {-3, ..., 3} { 
	\draw[fill = black, shift = {(0, 0)}] (\x / 2, \y / 2) circle(.3mm); 
	\draw[fill = black, shift = {(4.5, 0)}] (\x / 2, \y / 2) circle(.3mm); 
	\draw[fill = black, shift = {(9, 0)}] (\x / 2, \y / 2) circle(.3mm); 
	}}
		
\draw[shift = {(0, 0)}, very thick] (1.7, 0) -- (0, 0) -- (0, 0.5) -- (1.2, 1.7)
												(1.2, 1.7) -- (0, 0.5) -- (0, 1.7)
												(0, 1.7) -- (0, 0.5) -- (-0.6, 1.7)
												(-0.6, 1.7) -- (0, 0.5) -- (0, 0) -- (-1.7, 1.7)
												(-1.7, 1.7) -- (0, 0) -- (0, -0.5) -- (-1.7, -0.5)
												(-1.7, -0.5) -- (0, -0.5) -- (0, -1.7)
												(0, -1.7) -- (0, -0.5) -- (1.2, -1.7)
												(1.2, -1.7) -- (0, -0.5) -- (0, 0) -- (1.7, 0);
												
\draw (1.9, 0.66) node[] {{\small $\mathcal{D}_1$}};
\draw (0.5, 1.9) node[] {{\small $\mathcal{D}_2$}};
\draw (-0.25, 1.9) node[] {{\small $\mathcal{D}_3$}};
\draw (-1, 1.9) node[] {{\small $\mathcal{D}_4$}};
\draw (-1.9, 0.5) node[] {{\small $\mathcal{D}_5$}};
\draw (-1.9, -1.25) node[] {{\small $\mathcal{D}_6$}};
\draw (0.5, -1.9) node[] {{\small $\mathcal{D}_7$}};
\draw (1.9, -0.66) node[] {{\small $\mathcal{D}_8$}};

\draw[shift = {(4.5, 0)}, very thick] (1.7, 0) -- (0.5, 0) -- (1.7, 1.2)
												(1.7, 1.2) -- (0.5, 0) -- (0, 0) -- (0, 1.7)
												(0, 1.7) -- (0, 0) -- (-0.75 - .15, 1.7)
												(-0.75 - .15, 1.7) -- (0, 0) -- (-1.7, 1.7)
												(-1.7, 1.7) -- (0, 0) -- (-1.7, 0)
												(-1.7, 0) -- (0, 0) -- (0, -1.7)
												(0, -1.7) -- (0, 0) -- (0.5, 0) -- (1.7, -1.2)
												(1.7, -1.2) -- (0.5, 0) -- (1.7, 0);

\draw[shift = {(9, 0)}, very thick] (1.7, 0) -- (0, 0) -- (1.7, 1.7)
												(1.7, 1.7) -- (0, 0) -- (0, 1.7)
												(0, 1.7) -- (0, 0) -- (-0.5, 0.5) -- (-1.1, 1.7)
												(-1.1, 1.7) -- (-0.5, 0.5) -- (-1.7, 1.7)
												(-1.7, 1.7) -- (-0.5, 0.5) -- (-1.7, 0.5)
												(-1.7, 0.5) -- (-0.5, 0.5) -- (0, 0) -- (0, -1.7)
												(0, -1.7) -- (0, 0) -- (1.7, -1.7)
												(1.7, -1.7) -- (0, 0) -- (1.7, 0);

\draw[<->, very thick] (-1, -2.5) -- (10, -2.5);
\draw[shift = {(0, -2.5)}] (0, -0.2) node[below] {0} -- (0, 0.2);
\draw[shift = {(4.5, -2.5)}] (0, -0.2) node[below] {1} -- (0, 0.2);
\draw[shift = {(9, -2.5)}] (0, -0.2) node[below] {$\infty$} -- (0, 0.2);
												
\end{tikzpicture}
\caption{}
\label{fig_pctb}
\end{figure}

\begin{table}[!ht]
\begin{center}
\begin{tabular}{ r | c c c c c c c c}
   & $\mathcal{D}_1$ & $\mathcal{D}_2$ & $\mathcal{D}_3$ & $\mathcal{D}_4$  & $\mathcal{D}_5 $ & $ \mathcal{D}_6 $ & $ \mathcal{D}_7 $ & $ \mathcal{D}_8$ \\ \hline

  \multirow{2}{*}{\scriptsize $D_{[0], (0, 1)} $} & {\scriptsize $ 0 $} & {\scriptsize $ 0 $ } & {\scriptsize  $ 0 $ } & {\scriptsize  $ 0 $ } & {\scriptsize  $ 0 $ } & {\scriptsize  $ 0 $ } & {\scriptsize  $ 0 $ } & {\scriptsize  $ 0 $}  \\ 
& {\scriptsize $ (0, 1) $} & {\scriptsize $ (0,1) $ } & {\scriptsize  $ (1,1) $ } & {\scriptsize  $ (1,1) $ }  & {\scriptsize  $ (0, 0) $ } & {\scriptsize  $ (0, 0) $ } & {\scriptsize  $ (0, 0) $ } & {\scriptsize  $ (0, 0) $} \\ \hline

  \multirow{2}{*}{\scriptsize $D_{[0], (0, 0)}$ } & {\tiny  $[0] \vslimminus [1]$ } & {\scriptsize  $0$ } & {\scriptsize  $0$ } & {\tiny  $[0] \vslimminus [\infty]$ } & {\tiny  $[0] \vslimminus [\infty]$ } & {\scriptsize  $ 0 $ } & {\scriptsize  $ 0 $ } & {\tiny  $[0] \vslimminus [1]$} \\ 
  & {\scriptsize  $(1,\slimminus1)$ } & {\scriptsize  $(0, 0)$ } & {\scriptsize  $(0, 0)$ } & {\scriptsize  $(\slimminus2,\slimminus1)$ } & {\scriptsize  ${(0,1)}$ } & {\scriptsize  $ (0, 0) $ } & {\scriptsize  $ (0, 0) $ } & {\scriptsize  ${(1, 1)}$} \\ \hline
  
  \multirow{2}{*}{\scriptsize $D_{[0], (0, \slimminus1)}$ } & {\scriptsize  $0$ } & {\scriptsize  $0$ } & {\scriptsize  $0$ } & {\scriptsize  $0$ }  & {\scriptsize  $0$ } & {\scriptsize  $0$ } & {\scriptsize  $0$ } & {\scriptsize  $0$}\\ 
  & {\scriptsize  $(0, 0)$ } & {\scriptsize  $(0, 0)$ } & {\scriptsize  $(0, 0)$ } & {\scriptsize  $(0, 0)$ }  & {\scriptsize  ${(\slimminus1, \slimminus1)}$ } & {\scriptsize  ${(\slimminus1, \slimminus1)}$ } & {\scriptsize  ${(0, \slimminus1)}$ } & {\scriptsize  ${(0, \slimminus1)}$}\\ \hline
  
  \multirow{2}{*}{\scriptsize $D_{[1], (1, 0)}$ } & {\scriptsize  $0$ } & {\scriptsize  $0$ } & {\scriptsize  $0$ } & {\scriptsize  $0$ } & {\scriptsize  $0$ } & {\scriptsize  $0$ } & {\scriptsize  $0$ } & {\scriptsize  $0$}\\ 
  & {\scriptsize  ${(1, 0)}$ } & {\scriptsize  ${(1, 0)}$ } & {\scriptsize  $(0, 0)$ } & {\scriptsize  $(0, 0)$ } & {\scriptsize  $(0, 0)$ } & {\scriptsize  $(0, 0)$ } & {\scriptsize  ${(1, 0)}$ } & {\scriptsize  ${(1, 0)}$}\\ \hline

  \multirow{2}{*}{\scriptsize $D_{[1], (0, 0)}$ } & {\scriptsize  $0$ } & {\tiny  $[1]\vslimminus[0]$ } & {\tiny  $[1]\vslimminus[0]$ } & {\tiny  $[1]\vslimminus[\infty]$ } & {\tiny  $[1]\vslimminus[\infty]$ } & {\tiny  $[1]\vslimminus[0]$ } & {\scriptsize  $[1]\vslimminus[0]$ } & {\scriptsize  $0$} \\ 
  & {\scriptsize  $(0, 0)$ } & {\scriptsize  ${(\slimminus1, 1)}$ } & {\scriptsize  ${(1, 1)}$ } & {\scriptsize  ${(\slimminus1, 0)}$ } & {\scriptsize  ${(\slimminus1, 0)}$ } & {\scriptsize  ${(\slimminus1, \slimminus1)}$ } & {\scriptsize  ${(\slimminus1, \slimminus1)}$ } & {\scriptsize  $(0, 0)$} \\ \hline
  
  \multirow{2}{*}{\scriptsize $D_{[\infty], (\slimminus1, 1)}$ } & {\scriptsize  $0$ } & {\scriptsize  $0$ } & {\scriptsize  $0$ } & {\scriptsize  $0$ } & {\scriptsize  $0$ } & {\scriptsize  $0$ } & {\scriptsize  $0$ } & {\scriptsize  $0$} \\ 
  & {\scriptsize  $(0, 0)$ } & {\scriptsize  $(0, 0)$ } & {\scriptsize  ${(\slimminus1, 0)}$ } & {\scriptsize  ${(\slimminus1, 0)}$ } & {\scriptsize  ${(\slimminus1, 0)}$ } & {\scriptsize  ${(\slimminus1, 0)}$ } & {\scriptsize  $(0, 0)$ } & {\scriptsize  $(0, 0)$} \\ \hline
  
  \multirow{2}{*}{\scriptsize $D_{[\infty], (0, 0)}$ } & {\tiny  $[\infty]\vslimminus[1]$ } & {\tiny  $[\infty]\vslimminus[0]$ } & {\tiny  $[\infty]\vslimminus[0]$ } & {\scriptsize  $0$ } & {\scriptsize  $0$ } & {\tiny  $[\infty] \vslimminus [0]$ } & {\tiny  $[\infty] \vslimminus [0]$ } & {\tiny  $[\infty] \vslimminus [1]$}\\
  & {\scriptsize  ${(1, 0)}$ } & {\scriptsize  ${(0, 1)}$ } & {\scriptsize  ${(2, 1)}$ } & {\scriptsize  $(0, 0)$ } & {\scriptsize  $(0, 0)$ } & {\scriptsize  ${(0, \slimminus1)}$ } & {\scriptsize  ${(0, \slimminus1)}$ } & {\scriptsize  ${(1, 0)}$}\\ \hline
\end{tabular}
\end{center}
\caption{Torus invariant divisors on $TV(\mathcal{S})$}
\label{tab_pctbdivisors}
\end{table}

Because every maximal-dimensional cone of $\textrm{tail}(\mathcal{S})$ is marked, $TV(\mathcal{S})$ has no horizontal curves. Let $\tau_{i, j, y}$ be the wall of $\mathcal{S}_y$ realized as the intersection between $\mathcal{D}_i$ and $\mathcal{D}_j$ (if such a wall exists). Using Proposition \ref{prop_linequiv}, we see that the numerical equivalence class of $C_{\tau_{i, j, y}, y}$ only depends on $i$ and $j$; to save space, we abbreviate $[C_{\tau_{i, j, y}, y}]$ as $C_{i, j}$. 

As an example of a calculation, consider $C_{1, 2}$ and the $T$-invariant divisor $D_{[0], (0, 0)}$ with Cartier support function $\{h_{\mathcal{D}, y}\}$. Using notation from Section \ref{section_vertical_curves}, $n_{\tau, \mathcal{D}_2} = (0, 1)$. The relevant linear parts of the Cartier support function are $\textrm{lin}h_{\mathcal{D}_1, 0} = -(1, -1) \in M$ and $\textrm{lin}h_{\mathcal{D}_2, 0} = (0, 0) \in M$. The intersection can then be calculated using Equation \ref{equation_vertintersect}
\[
\langle D_{[0], (0, 0)}, C_{1,2} \rangle = 1^{-1}\langle (-1, 1) - (0, 0), (0, 1)\rangle = 1
\]

The complete list of intersections is in Table \ref{tab_pctbintersection}. The canonical divisor is also listed; it can be expressed as a sum of the vertical divisors using the formula from (\cite{tid}, Theorem 3.21).

\begin{table}[!ht]
\begin{center}
\begin{tabular}{ r | c c c c c c c c c c c c}
   & {\tiny $C_{1, 2}$} & {\tiny $C_{2, 3}$} & {\tiny $C_{3, 4}$} & {\tiny $C_{4, 5}$} & {\tiny $C_{5, 6}$} & {\tiny $C_{6, 7}$} & {\tiny $C_{7, 8}$} & {\tiny $C_{8, 1}$}  &  {\tiny $C_{1, 4}$} & {\tiny $C_{5, 8}$} & {\tiny $C_{2, 7}$} & {\tiny $C_{3, 6}$} \\ \hline

  {\scriptsize $D_{[0], (0, 1)} $} & 0 & -1 & 0& 1 & 0& 0& 0&1 & -1& 0& 1&1 \\ \hline

  {\scriptsize $D_{[0], (0, 0)}$ } & 1 & 0& 1& -2  & 1& 0& 1&-2& 3& 1& 0&0 \\ \hline

  {\scriptsize $D_{[0], (0, -1)}$ } & 0& 0& 0& 1   & 0& 1& 0&1 & 0 & 1& 1&1 \\ \hline

  {\scriptsize $D_{[1], (1, 0)}$ } & 0& 1& 0& 0    & 0& 1& 0&0 & 1 & 1& 0&0 \\ \hline

  {\scriptsize $D_{[1], (0, 0)}$ } & 1& -2& 1& 0   & 1& 0& 1&0 & 1 & 1& 2&2 \\ \hline

  {\scriptsize $D_{[\infty], (-1, 1)}$ } & 0&  1& 0&0&0& 1& 0&0 & 1 & 1& 0&0 \\ \hline

  {\scriptsize $D_{[\infty], (0, 0)}$ } &  1& -2& 1&0&1& 0& 1&0 & 1 & 1& 2&2 \\ \hline
  
  {\scriptsize $K_{X}$ } & -2 & 2&-2&0&-2& -2& -2&0 & -4 & -4& -4&-4 \\ \hline
\end{tabular}
\end{center}
\caption{Intersections of divisors and curves on $TV(\mathcal{S})$}
\label{tab_pctbintersection}
\end{table}
\subsection{Example 2}
Let $\sigma_1, \sigma_2, \sigma_3$ be the cones
\begin{align*}
\sigma_1 &= \mathbb{Q}_{\geq 0} \cdot(1, 0) +  \mathbb{Q}_{\geq 0} \cdot(0, 1)\\
\sigma_2 &= \mathbb{Q}_{\geq 0} \cdot(0, 1) +  \mathbb{Q}_{\geq 0} \cdot(-1, -1)\\
\sigma_3 &= \mathbb{Q}_{\geq 0} \cdot(1, 0) +  \mathbb{Q}_{\geq 0} \cdot(-1, -1))
\end{align*}
and let $\mathcal{S}$ be the divisorial fan on $\mathbb{P}^1$ having the following maximal $p$-divisors
\begin{align*}
\mathcal{D}_1 &= ((2/3, 1/2) + \sigma_1)[0] + ((-2/3, -1/2) + \sigma_1)[1] + \emptyset[\infty]\\
\mathcal{D}_2 &= ((2/3, 1/2) + \sigma_2)[0] + ((-2/3, -1/2) + \sigma_2)[1] + ((-1, -1) + \sigma_2)[\infty]\\
\mathcal{D}_3 &= ((2/3, 1/2) + \sigma_3)[0] + ((-2/3, -1/2) + \sigma_3)[1] + ((-1, -1) + \sigma_3)[\infty]\\
\mathcal{D}_4 &= \emptyset[0] + \emptyset[1] + ((-1, -1) + \sigma_1)[\infty]
\end{align*}

\begin{figure}[!ht]
\centering
\begin{tikzpicture}

\path[shift = {(0, 0)}, fill = gray1] (.6 * 2.4, .6 * .5) -- (.6 * .68, .6 * .5) -- (.6 * .68, .6 * 2.4) -- (.6 * 2.4, .6 * 2.4) -- cycle;
\path[shift = {(0, 0)}, fill = gray3] (.6 * .68, .6 * 2.4) -- (.6 * .68, .6 * .5) -- (.6 * -2.24, .6 * -2.4) -- (.6 * -2.4, .6 * -2.4) -- (.6 * -2.4, .6 * 2.4) -- cycle;
\path[shift = {(0, 0)}, fill = gray5] (.6 * -2.24, .6 * -2.4) -- (.6 * .68, .6 * .5) -- (.6 * 2.4, .6 * .5) -- (.6 * 2.4, .6 * -2.4) -- cycle;

\draw[shift = {(4.5, 0)}, fill = gray1] (.6 * 2.4, .6 * -.5) -- (.6 * -.68, .6 * -.5) -- (.6 * -.68, .6 * 2.4) -- (.6 * 2.4, .6 * 2.4) -- cycle;
\draw[shift = {(4.5, 0)}, fill = gray3] (.6 * -.68, .6 * 2.4) -- (.6 * -.68, .6 * -.5) -- (.6 * -2.4, .6 * -2.24) -- (.6 * -2.4, .6 * -2.4) -- (.6 * -2.4, .6 * 2.4) -- cycle;
\draw[shift = {(4.5, 0)}, fill = gray5] (.6 * -2.4, .6 * -2.24) -- (.6 * -.68, .6 * -.5) -- (.6 * 2.4, .6 * -.5) -- (.6 * 2.4, .6 * -2.4) -- (.6 * -2.4, .6 * -2.4) -- cycle;
												
\draw[shift = {(9, 0)}, fill = gray7] (.6 * 2.4, .6 * -1) -- (.6 * -1, .6 * -1) -- (.6 * -1, .6 * 2.4) -- (.6 * 2.4, .6 * 2.4) -- cycle;
\draw[shift = {(9, 0)}, fill = gray3] (.6 * -1, .6 * 2.4) -- (.6 * -1, .6 * -1) -- (.6 * -2.4, .6 * -2.4) -- (.6 * -2.4, .6 * -2.4) -- (.6 * -2.4, .6 * 2.4) -- cycle;
\draw[shift = {(9, 0)}, fill = gray5] (.6 * -2.4, .6 * -2.4) -- (.6 * -1, .6 * -1) -- (.6 * 2.4, .6 * -1) -- (.6 * 2.4, .6 * -2.4) -- cycle;

\foreach \x in {-2, ..., 2} { \foreach \y in {-2, ..., 2} { 
	\draw[fill = black, shift = {(0, 0)}] (\x * .6, \y * .6) circle(.4mm); 
	\draw[fill = black, shift = {(4.5, 0)}] (\x *.6, \y *.6) circle(.4mm); 
	\draw[fill = black, shift = {(9, 0)}] (\x *.6, \y *.6) circle(.4mm); 
	}}
		
\draw[shift = {(0, 0)}, very thick] (.6 * 2.5, .6 * .5) -- (.6 * .68, .6 * .5) -- (.6 * .68, .6 * 2.5)
												(.6 * .68, .6 * 2.5) -- (.6 * .68, .6 * .5) -- (.6 * -2.34, .6 * -2.5)
												(.6 * -2.34, .6 * -2.5) -- (.6 * .68, .6 * .5) -- (.6 * 2.5, .6 * .5);

\draw[shift = {(4.5, 0)}, very thick] (.6 * 2.5, .6 * -.5) -- (.6 * -.68, .6 * -.5) -- (.6 * -.68, .6 * 2.5)
												(.6 * -.68, .6 * 2.5) -- (.6 * -.68, .6 * -.5) -- (.6 * -2.5, .6 * -2.34)
												(.6 * -2.5, .6 * -2.34) -- (.6 * -.68, .6 * -.5) -- (.6 * 2.5, .6 * -.5);
												
\draw[shift = {(9, 0)}, very thick] (.6 * 2.5, .6 * -1) -- (.6 * -1, .6 * -1) -- (.6 * -1, .6 * 2.5)
												(.6 * -1, .6 * 2.5) -- (.6 * -1, .6 * -1) -- (.6 * -2.5, .6 * -2.5)
												(.6 * -2.5, .6 * -2.5) -- (.6 * -1, .6 * -1) -- (.6 * 2.5, .6 * -1);

\draw (1.7, 1) node[] {{\small $\mathcal{D}_1$}};
\draw (-0.5, 1.7) node[] {{\small $\mathcal{D}_2$}};
\draw (0.5, -1.7) node[] {{\small $\mathcal{D}_3$}};
\draw[shift = {(9, 0)}] (1.7, 1) node[] {{\small $\mathcal{D}_4$}};

\draw[<->, very thick] (-1, -2.5) -- (10, -2.5);
\draw[shift = {(0, -2.5)}] (0, -0.2) node[below] {0} -- (0, 0.2);
\draw[shift = {(4.5, -2.5)}] (0, -0.2) node[below] {1} -- (0, 0.2);
\draw[shift = {(9, -2.5)}] (0, -0.2) node[below] {$\infty$} -- (0, 0.2);
												
\end{tikzpicture}
\caption{The divisorial fan $\mathcal{S}$}
\end{figure}
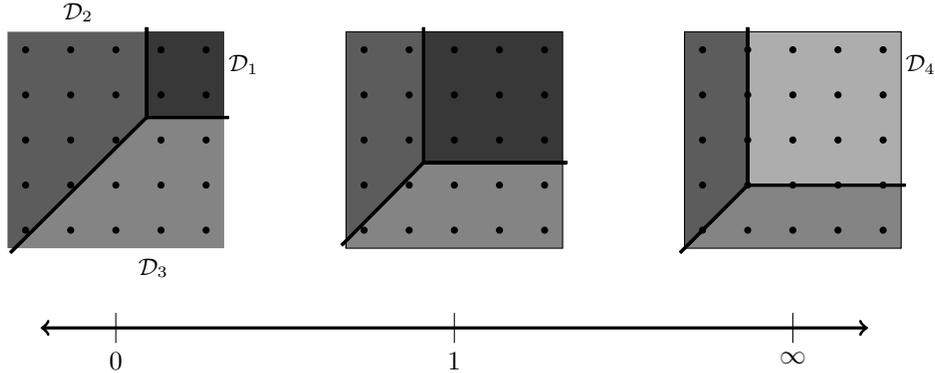

The $T$-variety corresponding to $\mathcal{S}$ is a deformation of $\mathbb{P}^3$. The $T$-invariant divisors and their intersections with $T$-invariant curves are encoded in Table \ref{tab_defp3divisors} and $\ref{tab_defp3intersections}$ respectively, using the same notation as in the previous example.
\begin{table}[!ht]
\begin{center}
\begin{tabular}{ r | c c c c}
   & $\mathcal{D}_1$ & $\mathcal{D}_2$ & $\mathcal{D}_3$ & $\mathcal{D}_4$  \\ \hline

  \multirow{2}{*}{$D_{[0], (\nicefrac{2}{3}, \nicefrac{1}{2})}$} & $\nicefrac{1}{6}[0]$ & $\nicefrac{5}{18}[0]\vslimminus\nicefrac{1}{9}[1]\vslimminus\nicefrac{1}{6}[\infty]$ & $\nicefrac{1}{4}[0] \vslimminus \nicefrac{1}{12}[1]\vslimminus\nicefrac{1}{6}[\infty]$ & $0$  \\
  & $(0, 0)$ & $(\nicefrac{-1}{6}, 0)$ & $(0, \nicefrac{-1}{6})$ & $(0, 0)$ \\ \hline
  
	\multirow{2}{*}{$D_{[1], (\nicefrac{-2}{3}, \nicefrac{-1}{2})}$} & $\nicefrac{1}{6}[1]$ & $\nicefrac{1}{9}[0] \vslimplus \nicefrac{1}{18}[1] \vslimminus \nicefrac{1}{6}[\infty]$ & $\nicefrac{1}{12}[0] \vslimplus \nicefrac{1}{12}[1] \vslimminus \nicefrac{1}{6}[\infty]$ & $0$  \\
	& $(0, 0)$ & $(\nicefrac{-1}{6}, 0)$ & ${(0, \nicefrac{-1}{6})}$ & ${(0, 0)}$  \\ \hline

	\multirow{2}{*}{$D_{[\infty], (\slimminus1, \slimminus1)}$} & $0$ & $\nicefrac{2}{3}[0] \vslimminus \nicefrac{2}{3}[1]$ & $\nicefrac{1}{2}[0] \vslimminus \nicefrac{1}{2}[1]$ & $[\infty]$  \\
	& $(0, 0)$ & $(\slimminus1, 0)$ & ${(0, \slimminus1)}$ & ${(0, 0)}$  \\ \hline

	\multirow{2}{*}{$D_{\mathbb{Q}_{\geq 0} \cdot (1, 0)}$} & $\nicefrac{-2}{3}[0] \vslimplus \nicefrac{2}{3}[1]$ & $0$ & $\nicefrac{-1}{6}[0] \vslimplus \nicefrac{1}{6}[1]$ & $[\infty]$  \\
	& $(0, 0)$ & $(\nicefrac{-1}{6}, 0)$ & ${(0, \nicefrac{-1}{6})}$ & ${(0, 0)}$  \\ \hline

	\multirow{2}{*}{$D_{\mathbb{Q}_{\geq 0} \cdot (0, 1)}$} & $\nicefrac{-1}{2}[0] \vslimplus \nicefrac{1}{2}[1]$ & $\nicefrac{1}{6}[0] \vslimminus \nicefrac{1}{6}[1]$ & $0$ & $[\infty]$  \\
	& $(0, 1)$ & $(\slimminus1, 1)$ & ${(0, 0)}$ & ${(0, 1)}$  \\ \hline
	
	\end{tabular}
	\caption{Torus invariant divisors on $TV(\mathcal{S})$}
	\label{tab_defp3divisors}
	\end{center}
	\end{table}

\begin{table}[!ht]
\begin{center}
\begin{tabular}{ r | c c c c c}
   & $C_{\tau_{1, 2}}$ & $C_{\tau_{2,3}}$ & $C_{\tau_{1, 3}}$ & $C_{\sigma_1}$\\ \hline

  $D_{[0], (\nicefrac{2}{3}, \nicefrac{1}{2})}$ & \nicefrac{1}{6}& \nicefrac{1}{6}& \nicefrac{1}{6}&  \nicefrac{1}{6}  \\ \hline
$D_{[1], (\nicefrac{-2}{3}, \nicefrac{-1}{2})}$ & \nicefrac{1}{6}& \nicefrac{1}{6}& \nicefrac{1}{6}&  \nicefrac{1}{6}  \\ \hline
$D_{[\infty], (\slimminus1, \slimminus1)}$      & 1& 1& 1&  1  \\ \hline
$D_{\mathbb{Q}_{\geq 0} \cdot (1, 0)}$          & 1& 1& 1&  1  \\ \hline
$D_{\mathbb{Q}_{\geq 0} \cdot (0, 1)}$          & 1& 1& 1&  1  \\ \hline
\end{tabular}
\end{center}
\caption{Intersections on $TV(\mathcal{S})$}
\label{tab_defp3intersections}
\end{table}

\end{document}